\documentclass[a4paper, 11pt]{article}
\usepackage{latexsym,bm}
\usepackage{amssymb,amsmath,amsthm}
\usepackage{enumerate, enumitem}
\usepackage{graphicx}
\usepackage[english,german]{babel}
\usepackage{color}
\usepackage[colorlinks=true]{hyperref}
\hypersetup{linkcolor=blue,urlcolor=red,citecolor=red}
\usepackage{authblk}

\allowdisplaybreaks[4]

\newtheorem{theorem}{Theorem}[section]
\newtheorem{lemma}[theorem]{Lemma}
\newtheorem{remark}[theorem]{Remark}
\newtheorem{proposition}[theorem]{Propositon}
\newtheorem{definition}[theorem]{Definition}
\newtheorem{corollary}[theorem]{Corollary}
\numberwithin{equation}{section}
\makeatletter

\def\email#1{Email: #1}
\def\emails#1{Emails: #1}

\makeatother
\def\keywords#1{
    \par\medskip
    \noindent\textbf{Keywords.} #1
}
\def\subjclass#1{{
    \renewcommand{\thefootnote}{}%
    \footnote{
        \emph{Mathematics Subject Classification (2020):} #1
    }
}}
\title{Periodic propagation of singularities for heat equations with  time delay}

\author[1]{
Gengsheng Wang
}
\author[2,3]{
Huaiqiang Yu
}
\author[1,4]{
Yubiao Zhang
}

\affil[1]{
Center for Applied Mathematics and KL-AAGDM,
Tianjin University, 300072 Tianjin,  China.
\emails{
    \texttt{wanggs62@yeah.net},
    \texttt{yubiao\_zhang@yeah.net}
}
}
\affil[2]{
School of Mathematics and KL-AAGDM,
Tianjin University,
300350 Tianjin,  China.
\email{\texttt{huaiqiangyu@yeah.net}
}
}
\affil[3]{
Radon Institute for Computational and Applied Mathematics,
Austrian Academy of Sciences,
Altenbergerstraße 69, Linz, 4040, Upper Austria, Austria.
}

\affil[4]{
Chair for Dynamics, Control, Machine Learning  and Numerics,
Department of Mathematics,
Friedrich-Alexander-Universit\"{a}t Erlangen-N\"{u}rnberg,
91058 Erlangen, Germany.
}

\date{}

\begin{document}
\selectlanguage{english}
\maketitle

\begin{abstract}

This paper presents two remarkable phenomena associated with the heat equation with a time delay: namely, the propagation of singularities and periodicity. These are manifested through a distinctive mode of propagation of singularities in the solutions. Precisely, the singularities of the solutions propagate periodically in a bidirectional fashion along the time axis. Furthermore, this propagation occurs in a stepwise manner. More specifically, when propagating in the positive time direction, the order of the joint derivatives of the solution increases by 2 for each period; conversely, when propagating in the reverse time direction, the order of the joint derivatives decreases by 2 per period. Additionally, we elucidate the way in which the initial data and historical values impact such a propagation of singularities.

The phenomena we have discerned not only corroborate the pronounced differences between heat equations with and without time delay but also vividly illustrate the substantial divergence between the heat equation with a time delay and the wave equation, especially when viewed from the point of view of singularity propagation.

%This paper presents the propagation of singularities for solutions of heat equations with time delay. Such propagation is along the time direction, exhibits the characteristics of discreteness, periodicity, and attenuation, and shows the hyperbolic aspects of heat equations with time delay. This not only shows wide differences between heat equations with and without time delay but is also quite unusual compared to classical wave equations. Besides, it may be of interest that heat equations with time delay are shown to have periodicity.
\end{abstract}



\section{Introduction}%\label{yu-section-1}

\subsection{Equation under study}
Let $\mathbb{R}^+:=(0,+\infty)$ and $\mathbb N^+:=\{1,2,\ldots\}$.
Let $\Omega \subset \mathbb R^n$    ($n\in \mathbb N^+$)     be a bounded domain with a $C^2$-boundary $\partial\Omega$.
This paper studies  the following heat equation with a time delay:
\begin{equation}\label{20220510-yb-OriginalSystemWithTimeDelay}
	\begin{cases}
	 	\partial_t  y(t,x) = \Delta y(t,x) + a y(t-\tau,x)    &(t,x)\in\mathbb{R}^+\times\Omega,
	 	\\
	 	y(t,x)=0  &(t,x)\in\mathbb{R}^+\times\partial\Omega, \\
	 	y(0,x)  =  y_0(x)  &\,~~~~~~x\in \Omega,\\
	 	y(t,x) = \phi(t,x) &(t,x)\in(-\tau,0)\times\Omega,
	\end{cases}
\end{equation}
    where   $\tau>0$
    and $a\in \mathbb R \setminus\{0\}$ are given constants, and
    $y_0\in L^2(\Omega)$ and $\phi \in L^2( (-\tau,0) \times \Omega )$ are the initial value and the historical value, respectively.
%    The solution of equation \eqref{20220510-yb-OriginalSystemWithTimeDelay}
%    is usually defined in $[0,+\infty)$.
%    It is well known that equation \eqref{20220510-yb-OriginalSystemWithTimeDelay} has a unique solution
% in $C([0,+\infty);L^2(\Omega))$ (see, for example, \cite[Theorem 2.1]{Nakagiri-1988}).

\subsection{History, motivation and aim}
Differential equations with time delay constitute an important subject in the domain of differential equations. Numerous natural phenomena emerging from physics, chemistry, biology, economics, and other fields hinge not merely on the current state but also on certain past events (see, for example, \cite{Batkai-Piazzera-2005, Diekmann-GVW-Springer-1995, Erneux-2009, Hale-Verduyn-Springer-1993, Kuang-1993, Wu-1996} and the references therein). Owing to the aforementioned reasons, research on differential equations with time delay boasts a considerably long history.

To the best of our knowledge, the related research originated from the study of the well-posedness of linear ordinary differential equations with delays during the 1960s and 1970s. Its foundation lies in the state space theory (also known as the so-called semigroup theory, or the evolution equation approach) that was developed in works such as \cite{Bernier-Manitius-1978, Burns-Herdman-Stech-1983, Delfour-1977, Delfour-Manitius-1980, Manitius-1980}. Subsequently, this theory was promptly extended to infinite-dimensional scenarios, particularly to partial differential equations with delays. For the latter, we refer the reader to \cite{Batkai-Piazzera-2005, Kunisch-Schappacher-1980, Kunisch-Schappacher-1983, Nakagiri-1981, Nakagiri-1987, Nakagiri-1988, Pinto-Poblete-Sepulveda, Vinter-1978} and the references therein.

The state space theory is convenient for analyzing the global regularity and dynamic behavior of equations with delay because this theory can transform equations with delay into a system of linear evolution equations without delay. Then, we can apply the theory of a one-parameter semigroup to obtain many useful results for the equations with delay.

%This theory/method has greatly stimulated the %development of other topics related to equations with %delay. We refer the reader to \cite{Delfour-McCalla-%Mitter-1975, Nakagiri-1986, Delfour-1977, Gibson-%1983, Vinter-Kwong-1981} for optimal control %problems; \cite{Bartosiewicz-1984, Salamon-1984, %Rabah-Sklyar-2007, Nakagiri-Yamamoto-1989, Manitius-%1981, Manitius-Triggiani-1978, Khodja-Bouzidi-Dupaix-%Maniar-2014, Jeong-1993} for controllability; %\cite{Nakagiri-Yamamoto-1988, Nakagiri-Haruki-2003} %for the identification problem; and \cite{Yamamoto-%1987, Nicaise-Pignotti-2006, Richard-2003, Prieur-%Trelat-2018, Nakagiri-Yamamoto-2001, Jeong-1991} for %stability or stabilizability.

 The motivation for this research is as follows.
\begin{itemize}[leftmargin=4em]

\item[(m1)] We can recognize that equation \eqref{20220510-yb-OriginalSystemWithTimeDelay} may have some kind of  propagation of singularity from the following three aspects.

First,
 we consider  the following ODE with time-delay:
    \begin{align*}
        y'(t) = y(t-1), ~ t>0;~~ y(0) \in \mathbb R,
        ~ y|_{(-1,0)} \in L^2((-1, 0)).
    \end{align*}
One can easily find that  a solution $y(t)$ to the above equation  is not smooth  over $(0,1)$, provided that   the historical value $y|_{(-1,0)}$ is not smooth either.

Second, employing the state space method, we are able to rewrite equation \eqref{20220510-yb-OriginalSystemWithTimeDelay} as the subsequent autonomous system which consists of a heat equation coupled with a transport equation:
    \begin{align}\label{20250203-HybridSystemOfHeatAndTransport}
        \left\{\begin{array}{ll}
            \partial_t y(t)  -  \Delta_D y(t) = a z(t,1),~  t>0,\\
          \partial_t z(t,s) + \partial_s z(t,s) = 0, ~t>0,~s\in(0,\tau),\\
          z(t,0)  = y(t) ~\text{in}~  L^2(\Omega), ~ t>0, \\
        y(0) = y_0 ~\text{and}~
        z(0,s) = \phi(-s) ~\text{in}~  L^2(\Omega), ~s\in(0,\tau),
        \end{array}
        \right.
    \end{align}
    where $\Delta_D$ is the Laplacian operator with the homogeneous Dirichlet boundary condition (see \eqref{20241224-yb-LaplacianWithDiricihletBoundaryCondition}). However,
the transport equation has the propagation of singularities along the characteristic lines.

Third, it has been widely recognized for a long time that heat equations with integral-type memory exhibit certain wavelike characteristics (see, for instance, \cite{Cattaneo-1958, Gurtin-Pipkin-1968} and the references therein). Notably, recently in \cite{Wang-Zhang-Zuazua-2022}, the authors demonstrated that the heat equation with integral type memory, when equipped with smooth kernels (in the zero-order term of the equation), features the propagation of singularities at zero velocity. Nevertheless, equation \eqref{20220510-yb-OriginalSystemWithTimeDelay} can be construed as a heat equation with integral memory, in which the memory kernel takes the form of a Dirac function. Consequently, investigating the propagation of singularities for equation \eqref{20220510-yb-OriginalSystemWithTimeDelay} is of considerable significance.

The above considerations prompt us to explore the following issues: Firstly, what sort of singularity propagation does equation \eqref{20220510-yb-OriginalSystemWithTimeDelay} exhibit? Secondly, in what ways do the historical values  and the initial data  influence the singularity propagation of equation \eqref{20220510-yb-OriginalSystemWithTimeDelay}?

    \item[(m2)]

   Although extensive research has been conducted on partial differential equations with time delay, in comparison to general partial differential equations, our comprehension of those with time delay remains considerably restricted. Particularly, to the best of our knowledge, quantitative investigations regarding the propagation of singularities for equation \eqref{20220510-yb-OriginalSystemWithTimeDelay} are virtually nonexistent. Given that the propagation of singularities is a characteristic typically associated with hyperbolic equations (specifically, the wave equation and the transport equation), its manifestation in the heat equation with delay is, in general, even more deserving of scrutiny.

    \item[(m3)] The propagation of singularities is intrinsically of great significance. This can be expounded upon from several perspectives as follows.
Firstly, the exploration of the propagation of singularities represents an essential subject matter (even one of the central themes) within the realm of the theory of partial differential equations/operators. Phenomena like unique continuation, as well as the propagation of wavefront sets. %, and the propagation of regularities can all find their roots in this topic (refer to, for example, \cite{Bouzouina-Robert-2002, Donnelly-Fefferman-1990, Hormander-1991, Ivrii-1998, Lin-1990, Zworski-2012}, among others).
Secondly, in a multitude of specific application scenarios, if our objective is to attain a comprehensive solution, we will invariably have to delve into certain facets of the propagation of singularities for the corresponding equations. For instance, this holds true when handling the controllability or observability of linear control systems (see \cite{Bardos-Lebeau-Rauch-1992, Nakagiri-Yamamoto-1989, WZZ-2021}, and the references therein).

   %\item[(m4)]
   %The spectrum of the generator of the autonomous %system \eqref{20250203-%HybridSystemOfHeatAndTransport} is approximately %distributed over an exponential curve along with %the negative real axis (a relevant reference %should be provided here).
  % Consequently, this generator generates  neither an %analytic semigroup nor a group. This aspect further %spurs us to uncover the properties (of equation %\eqref{20220510-yb-OriginalSystemWithTimeDelay}) %that distinguish it from those of heat equations %and wave equations.

\end{itemize}

   The purpose of this paper is as follows.
\begin{itemize}[leftmargin=4em]
 \item[(p1)] We are going to disclose novel phenomena pertaining to equation \eqref{20220510-yb-OriginalSystemWithTimeDelay}: namely, the propagation of singularities and periodicity. These are manifested through a distinctive propagation of singularities for the solutions and vividly illustrate the differences and similarities between equation \eqref{20220510-yb-OriginalSystemWithTimeDelay} and the heat/wave equations.

    \item[(p2)] We demonstrate that the singularities of the solutions propagate periodically in a bidirectional manner along the time axis. Moreover, this propagation occurs in a stepwise fashion. When propagating in the positive time direction, the order of the joint derivative of the solution increases by 2 for each period; conversely, when propagating in the reverse time direction, the order of the joint derivative decreases by 2 per period. (Refer to Definition \ref{definition1.1-1-6}.)

   \item[(p3)] We will clarify how the initial data and historical values collectively affect the propagation of the aforementioned singularities.
\end{itemize}

\subsection{Notation, subspaces and concepts}

We begin with the following notation that will be used frequently in this paper: Let $\mathbb{N}:=\{0,1,2,\ldots\}$ and $t^+:= \max\{t,0\}$. We write $\chi_K$ for the characteristic function of a set $K$, $\mathcal{L}(E,F)$ for the space of all linear bounded operators from a Banach space $E$ to another Banach space $F$, $\mathcal{L}(E)$ for the space $\mathcal{L}(E,E)$, and  $\mathcal{D}'(U)$ for the space of all distributions over a nonempty open subset $U$ of $\mathbb{R}^n$.

Now we introduce several spaces of functions that will be frequently used in our studies.
Given $s\in\mathbb{R}$ and a nonempty open subset $ U \subset \mathbb R^n$, we let $ H^s_{loc}(U)$ be the space of all distributions $f$ over $U$ such that $\widetilde{\rho f} \in H^s(\mathbb R^n)$ (with $\widetilde{\rho f}$ the zero extension of $\rho f$ to $\mathbb R^n$) for all $\rho \in C_0^{\infty}(U)$. Given
numbers $r,s\in\mathbb{N}$ and nonempty open subsets $I\subset \mathbb R $ and
$U \subset  \mathbb R^n$, we denote by $ H^{r,s}(I\times U) $ and $ H^{r,s}_{loc}(I\times U) $ the space $ H^r(I; H^s(U) )$ and the space of all functions $f$ for which $\rho f \in H^{r,s}(I\times U)$ for all $\rho \in C_0^{\infty}(I \times U)$, respectively.  Given $(t,x)\in \mathbb{R}\times \mathbb R^n$ and $r,s,p\in\mathbb{N}$, we define the following spaces:
\begin{equation}\label{yu-10-15-1}
H_{loc}^s(x) := \Big\{
    \text{all functions in}~    H^s_{loc}(U)
     ~\text{for some open neighborhood}~       U  ~\text{of}~       x          \Big\};
\end{equation}
\begin{eqnarray}\label{yu-10-15-2}
    H_{loc}^{r,s}(t,x):= \Big\{
        \text{all functions in}~    H^{r,s}_{loc}(I \times U)
         ~\text{for some open neighborhood}~      I \times  U  ~\text{of}~       (t,x)
    \Big\};
   \end{eqnarray}
\begin{align}\label{20241218-yb-JointLocalSobolevSpace}
     \widehat{H}^{(r,s),p}_{loc}(t,x) :=
    \bigcap_{\alpha,\beta\in\mathbb{N},\,2\alpha+\beta= p}
    H^{r+\alpha, s+\beta}_{loc}(t,x);
\end{align}
and
\begin{align}\label{yu-10-15-3}
\begin{array}{ll}
    H^{r,s}_{loc,+}(t,x) :=& \Big\{
        \text{all functions}~  f ~\text{in}~  H^{r}( (t, t+\delta); H^s(U) )
        \\
        &\quad\quad\quad
         ~\text{for some}~      \delta > 0     ~\text{and open neighborhood}~       U  ~\text{of}~       x
           \Big\};
\\
    H^{r,s}_{loc,-}(t,x) :=&  \Big\{
        \text{all functions}~  f ~\text{in}~
        H^{r}( (t-\delta, t); H^s(U) )
        \\
        &\quad\quad\quad
        ~\text{for some}~      \delta > 0
        ~\text{and open neighborhood}~
        U  ~\text{of}~       x
            \Big\}.
\end{array}
\end{align}

Next, we will introduce the following concepts that will be used to describe our main results.
\begin{definition}\label{definition1.1-1-6}
   The number $p$ in \eqref{20241218-yb-JointLocalSobolevSpace}
is referred to as the order of the joint derivatives of a function $g \in \widehat{H}^{(r,s),p}_{loc}(t,x)$ at the point $(t,x)$.
\end{definition}

%\color{blue} {
%\begin{definition}\label{definition1.2-1-6-w}
 %Let $g\in L^2_{loc}( [-\tau,+\infty) \times \Omega )$. Let $t_0 \in [-\tau, 0)$. We say that $g$ has
% the $\tau$-periodical jumping propagation of singularities
%along the time direction with  height $2$ and
% in the time sequence $\{ t_0 + j\tau \}_{ j\in\mathbb{N}^+ }$ (or $\{ t_0 + j\tau \}_{ j\in\mathbb{N}}$),
% if  for any $x_0\in\Omega$, there is a pair
%  $(r,s)\in\mathbb{N}\times\mathbb{N}$ such that  for each $j\in \mathbb N^+$ (or $j\in \mathbb N$),
%\begin{align}\label{1.6-1-8-w}
 %   g \not\in \widehat{H}^{(r,s),2j}_{loc}( t_0 + j\tau,x_0)
 %   ~\text{if and only if}~
%    g \not\in \widehat{H}^{(r,s), 2j+2 }_{loc} ( t_0 + (j+1)\tau,x_0).
    % \;\;\mbox{for each}\;\;j\in \mathbb N^+(\mbox{or} \;\mathbb N).
%\end{align}
% The aforementioned $\{ t_0 + j\tau \}_{ j\in\mathbb{N}^+ }$
%(or $\{ t_0 + j\tau \}_{ j\in\mathbb{N}}$) is called a periodic time
%sequence for the propagation of singularities of $g$.
%\end{definition} }\color{black}

Finally, due to the needs of our research,  we define
the solution of equation \eqref{20220510-yb-OriginalSystemWithTimeDelay}
in the following way:

\begin{definition}%\label{definition1.3-1-8-w}
Given $y_0\in L^2(\Omega)$ and $\phi \in L^2( (-\tau,0) \times \Omega )$,
    the solution of equation \eqref{20220510-yb-OriginalSystemWithTimeDelay}
    is  a function
$y: (-\tau,+\infty)\times \Omega\rightarrow  \mathbb{R}$ such that
\begin{itemize}[leftmargin=4em]
    \item [$(i)$] $ y|_{[0,+\infty)\times \Omega} \in C([0,+\infty); L^2(\Omega))
        \cap  L^2_{loc}([0,+\infty); H_0^1(\Omega))$;
  \item [$(ii)$] $y|_{(-\tau,0) \times \Omega} \in
        L^2( (-\tau,0) \times \Omega) )$;

    \item [$(iii)$]      $y$ solves equation \eqref{20220510-yb-OriginalSystemWithTimeDelay} in the mild sense (see Section \ref{section-ExplicitExpression}).
\end{itemize}
\end{definition}

One can easily check that the solution defined above exists and is unique.
{\it We denote this solution by $y(\cdot,\cdot;y_0,\phi)$ or $y(\cdot;y_0,\phi)$ which
corresponds to a real-valued function of time-space variables or a $L^2(\Omega)$-valued
function of the time variable. }
% It is clear that $ y(\cdot;y_0,\phi) \in L^2_{loc}((-\tau,+\infty);L^2(\Omega))$.

\subsection{Plan of this paper}
This paper is organized as follows. Section \ref{mainresults-1-7} displays our main results. Section \ref{section-LocalRegularity} presents several regularity properties for the classical heat equations, which will be used in the proofs of the main theorems. Section \ref{section-ExplicitExpression} provides an explicit expression of the solutions of equation \eqref{20220510-yb-OriginalSystemWithTimeDelay} in terms of the classical heat semigroup.
Section \ref{proof of main theorem}
 proves the main theorems. Section \ref{section-NumericalSimulations} offers numerical experiments that confirm the results of the first main result, Theorem \ref{20230510-yb-theorem-PropagationOfSingularitiesOnlyWithInitialdata}. Section \ref{section-yu-section-2-22} is the appendix.

\section{Main results}\label{mainresults-1-7}

The first main result of this paper describes the propagation of singularities in the time  sequence
 $\{j\tau\}_{j\in \mathbb N}$  for solutions of equation \eqref{20220510-yb-OriginalSystemWithTimeDelay} with the historical value
$\phi \in C_0^{\infty}((-\tau,0) \times \Omega)$.

\begin{theorem}\label{20230510-yb-theorem-PropagationOfSingularitiesOnlyWithInitialdata}
Let $y_0\in L^2(\Omega)$ and $\phi \in C_0^{\infty}((-\tau,0) \times \Omega)$.
%Then, the solution $y(\cdot,\cdot;y_0,\phi)$ has the $\tau$-periodically jumping periodic propagation of singularities along the time direction with the height $2$ and in $\{j\tau\}_{j\in \mathbb N}$. More precisely,
Then, for any $x_0\in \Omega$, $j\in \mathbb N$ and  $s,\alpha\in \mathbb N$  with $\alpha \leq j$, the following three statements are equivalent:
\begin{enumerate}
    \item [$(i)$] $y_0 \not\in H^s_{loc}(x_0)$;

    \item[$(ii)$] $y(\cdot,\cdot;y_0,\phi)  \not\in H_{loc}^{\alpha,s + 1 + 2(j-\alpha)}(j\tau,x_0)$;

    \item [$(iii)$] $ y(\cdot,\cdot;y_0,\phi)  \not\in
     \widehat{H}^{(0,s+1), 2j}_{loc}(j\tau,x_0)$.
\end{enumerate}
Furthermore, the solution $y(\cdot,\cdot;y_0,\phi)$ is smooth
in the set $\big( (-\tau,+\infty) \setminus \{j\tau\}_{j=0}^{+\infty} \big) \times \Omega$.
%outside the set $\{j\tau\}_{j=0}^{+\infty} \times \Omega$.
\end{theorem}

\begin{remark}\label{remark1.4-1.6-w}
Several notes on Theorem \ref{20230510-yb-theorem-PropagationOfSingularitiesOnlyWithInitialdata} are listed as follows:
\begin{enumerate}
  \item[$(i)$]
  In Theorem \ref{20230510-yb-theorem-PropagationOfSingularitiesOnlyWithInitialdata}, the statement $(i)$ is independent of $j\in \mathbb N$.
  Thus, the equivalence between $(i)$ and $(iii)$ in  Theorem \ref{20230510-yb-theorem-PropagationOfSingularitiesOnlyWithInitialdata} shows that
    equation \eqref{20220510-yb-OriginalSystemWithTimeDelay}
has the following special kind of propagation of singularities:
  the singularities of the solutions propagate periodically in a bidirectional fashion along the time axis. Furthermore, this propagation occurs in a stepwise manner. More specifically, when propagating in the positive time direction, the order of the joint derivatives of the solution increases by 2 for each period; conversely, when propagating in the reverse time direction, the order of the joint derivatives decreases by 2 per period.

  We may call the above as {\it the $\tau$-periodically jumping propagation of singularities along the time direction with  height $2$ and in the time sequence $\{j\tau\}_{j\in \mathbb N}$. }

  %Thus, the equivalence between $(i)$ and $(iii)$ in
%the above theorem, along with Definition \ref{definition1.2-1-6-w} {\color{red}(where $t_0=-\tau$ and $(r,s)$ is replaced by $(0,s+1)$)}  shows that $y(\cdot,\cdot;y_0,\phi)$ has the $\tau$-periodically jumping propagation of singularities along the time direction with  height $2$ and in the time sequence $\{j\tau\}_{j\in \mathbb N}$.

~~~~Theorem \ref{20230510-yb-theorem-PropagationOfSingularitiesOnlyWithInitialdata} also indicates that
when $\phi \in C_0^{\infty}((-\tau,0) \times \Omega)$,
the above mentioned propagation of singularities is valid only in the time sequence
$\{j\tau\}_{j \in \mathbb{N}}$.
% that corresponds to the case $t_0=-\tau$ in Definition \ref{definition1.2-1-6-w}.
The reason for this is that when  $\phi \in C_0^{\infty}((-\tau,0) \times \Omega)$, $\phi$ and $y_0$  have no mixing effect on $y(\cdot,\cdot;y_0,\phi)$ at $t=0$.

~~~~However, when $\phi\in L^2((-\tau,0)\times\Omega)$, the situation alters.
At this time, the periodic time sequence of the propagation of singularities of   $y(\cdot,\cdot;y_0,\phi)$ can be
 $\{ t_0 + j\tau \}_{ j\in\mathbb{N}}$ for any $t_0\in [-\tau,0)$.
What is more interesting is that the situations corresponding to $t_0=-\tau$ and $t_0\in (-\tau,0)$ are different. These will be presented in the next two theorems.

\item[$(ii)$] The propagation of singularities in Theorem \ref{20230510-yb-theorem-PropagationOfSingularitiesOnlyWithInitialdata} differs significantly from that seen in classical linear wave equations (for comprehensive discussions, see, for example, \cite{Bardos-Lebeau-Rauch-1992, Rauch-2012}), as well as heat equations with memory (see \cite{Wang-Zhang-Zuazua-2022}). It exhibits the following three distinct characteristics:
% The propagation of singularities in Theorem \ref{20230510-yb-theorem-PropagationOfSingularitiesOnlyWithInitialdata} differs significantly from that seen in classical linear wave equations (for comprehensive discussions, see, for example, \cite{Bardos-Lebeau-Rauch-1992, Rauch-2012}), as well as heat equations with memory (see \cite{Wang-Zhang-Zuazua-2022}). It exhibits the following three distinct characteristics:
\begin{itemize}[leftmargin=3em]
    \item[$(1)$] Singularities propagate along the time direction with zero velocity.
    \item[$(2)$] Singularities recur periodically, appearing once in each period, and while the solutions, given a regular history, remain smooth during the intervening time intervals.
    \item[$(3)$] As time progresses along the positive direction, the singularities decrease, accompanied by an increase of two orders in joint derivatives per period. While as time progresses along the negative direction, the singularities increase, accompanied by a reduction of two orders in joint derivatives per period.
\end{itemize}
To offer an intuitive image of the above propagation of singularities, we consider this scenario: a fading light that blinks at regular intervals, with its brightness gradually decreasing. In the context presented here, the light symbolizes the singularities of a solution; its blinking reflects the recurrence of these singularities, and the brightness of the light represents the degree of these singularities.

\item[$(iii)$] It might be surprising
 that the solution $y(\cdot;y_0,\phi)$ (as stated in Theorem \ref{20230510-yb-theorem-PropagationOfSingularitiesOnlyWithInitialdata}) shows greater smoothness at distant periodic lattice points along the time axis. However, this observation  aligns with an intuitive perspective: As the delay time $\tau$ approaches zero, the model \eqref{20220510-yb-OriginalSystemWithTimeDelay} turns back to the pure heat equation. In this case, the observed regularizing effect manifests itself in reproducing the inherent smoothness of solutions to the pure heat equation, even over small time intervals.

 \item[$(iv)$] The equivalence between (ii) and (iii) of Theorem \ref{20230510-yb-theorem-PropagationOfSingularitiesOnlyWithInitialdata} can be elucidated as follows: the interchangeability of time and space derivatives occurs in a fixed ratio to each other (specifically, each time derivative corresponds to two space derivatives). This parallel can be drawn to the analogous results found in classical heat equations.

~~~~Furthermore, the equivalence between (i) and (ii) of Theorem \ref{20230510-yb-theorem-PropagationOfSingularitiesOnlyWithInitialdata} still holds when the space
$H_{loc}^{\alpha,s + 1 + 2(j-\alpha)}(j\tau,x_0)$ is replaced by the
larger space $H_{loc,+}^{\alpha,s + 1 + 2(j-\alpha)}(j\tau,x_0)$ (defined by \eqref{yu-10-15-3}) which focuses only on the right neighborhoods of $j\tau$. This can be seen in the proof of this theorem.

\end{enumerate}
\end{remark}

The second main result of this paper presents the propagation of singularities in the time sequence $\{t_0 + j\tau\}_{j\in \mathbb N}$ (where $t_0\in(-\tau,0)$) for solutions of equation \eqref{20220510-yb-OriginalSystemWithTimeDelay} with  $\phi\in L^2((-\tau,0)\times\Omega)$.

\begin{theorem}\label{20230510-yb-theorem-PropagationOfSingularitiesForGeneralSolutions}
 Suppose that $y_0\in L^2(\Omega)$ and $\phi\in L^2((-\tau,0)\times\Omega)$. Let $t_0 \in (-\tau,0)$ and $x_0\in\Omega$.
 Let $r,s\in\mathbb{N}$. Then for each
   $j \in \mathbb N^+$, the following statements are equivalent:
\begin{enumerate}
  \item [$(i)$] $\phi \not\in H^{r,s}_{loc}(t_0, x_0)$;

  \item[$(ii)$] $y(\cdot,\cdot;y_0,\phi)  \not\in \widehat{H}^{(r,s), 2j}_{loc}(t_0+j\tau,x_0)$.
 \end{enumerate}
\end{theorem}

\begin{remark}%\label{remark1.6-1-7-w}
  Several notes regarding Theorem \ref{20230510-yb-theorem-PropagationOfSingularitiesForGeneralSolutions} are given as follows:
\begin{enumerate}
\item [$(i)$] Since the statement $(i)$ in Theorem \ref{20230510-yb-theorem-PropagationOfSingularitiesForGeneralSolutions} is independent of
   $j \in \mathbb N^+$,
    Theorem \ref{20230510-yb-theorem-PropagationOfSingularitiesForGeneralSolutions}
   %, along with Definition \ref{definition1.2-1-6-w},
   %and Definition \ref{definition1.3-1-8-w},
  shows that the solution $y(\cdot,\cdot;y_0,\phi)$ has the $\tau$-periodically jumping periodic propagation of singularities along the time direction with  height $2$ and in the time sequence $\{ t_0 + j\tau\}_{j\in \mathbb N}$ for each $t_0\in (-\tau,0)$, mentioned in the note $(i)$
  of     Remark \ref{remark1.4-1.6-w}.

\item [$(ii)$]  Theorem \ref{20230510-yb-theorem-PropagationOfSingularitiesForGeneralSolutions} only cares about case $t_0\in (-\tau,0)$. At this point, $y_0$ has no effect on the propagation of singularities for $y(\cdot,\cdot;y_0,\phi)$, and $\phi$ takes its place.

 \item [$(iii)$] Theorem \ref{20230510-yb-theorem-PropagationOfSingularitiesForGeneralSolutions} also shows how the singularities in the historical value at time $t_0\in (-\tau,0)$ dictate those of a solution at time instants $\{t_0 + j\tau\}_{j\in \mathbb N}$.
\end{enumerate}

\end{remark}

The third main result gives necessary and sufficient conditions (for
$y_0$ and $\phi$) on the propagation of singularities in
the periodic time sequence $\{j\tau\}_{j\in \mathbb N}$   for solutions of equation \eqref{20220510-yb-OriginalSystemWithTimeDelay}.
%with $\phi\in L^2((-\tau,0)\times\Omega)$.

\begin{theorem}\label{20240406-yubiao-Theorem-SingularitiesOfHistoricalValueAtEndpoints}
Suppose that $y_0\in L^2(\Omega)$ and $\phi\in L^2((-\tau,0)\times\Omega)$.
Let $x_0\in\Omega$ and $r,s\in\mathbb{N}$. Then for each
$j\in \mathbb N^+$,
the following statements are equivalent:
\begin{itemize}
\item[$(i)$]  The solution  $y(\cdot,\cdot;y_0, \phi) \not\in \widehat{H}^{(r,s),2j}_{loc}( -\tau + j \tau, x_0)$.
\item[$(ii)$]  At least one of the  following three statements is not true:
    \begin{itemize}
    \item[$(1)$]  $ \phi \in    H^{r,s}_{loc,+}(-\tau,x_0)  $  and  $ \phi \in    H^{r+1, s}_{loc,-}(0,x_0)  \cap   H^{r, s+2}_{loc,-}(0,x_0)$;
    \item[$(2)$]  The following functions
        \begin{align}\label{20241222-yb-DefinitionsOf-gk}
            g_0(\cdot) := y_0(\cdot),  ~\ldots,~
            g_k(\cdot) := a \partial_t^{k-1} \phi(-\tau, \cdot)   +  \Delta g_{k-1}(\cdot)
            ~\text{for}~
            k \in \mathbb N^+ \cap [0,r],
        \end{align}
    defined locally around $x_0$, have following regularities:
        \begin{align}\label{20241222-yb-RegularityOfAllDerivativesOfSolution-gk}
            g_k \in H^{s+1}_{loc}(x_0)   ~\text{for each integer}~   k \in [0,r];
        \end{align}
    \item[$(3)$] The following $r^{th}$-order compatibility conditions are verified:
        \begin{align*}
            \partial_t^{k} \phi(0,\cdot) = g_k(\cdot)
            ~\text{locally around}~   x_0
            ~\text{for each}~ k \in \mathbb N \cap [0,r].
        \end{align*}
    \end{itemize}
\end{itemize}

\end{theorem}

\begin{remark}%\label{yu-remark-11-22-2}
Some remarks on Theorem \ref{20240406-yubiao-Theorem-SingularitiesOfHistoricalValueAtEndpoints} are listed as follows:
\begin{enumerate}
    \item[$(i)$] Since the statement $(ii)$ in
    Theorem \ref{20240406-yubiao-Theorem-SingularitiesOfHistoricalValueAtEndpoints}  is independent of $j$,
     Theorem \ref{20240406-yubiao-Theorem-SingularitiesOfHistoricalValueAtEndpoints}   shows that $(ii)$ is a necessary and sufficient condition
for the periodic propagation of singularities $\tau$ -periodically jumping along the time direction with the height $2$ and
in the time sequence
  $\{j\tau\}_{j\in \mathbb N}$
  for the solution $y(\cdot,\cdot;y_0,\phi)$,  mentioned in the note $(i)$
of Remark \ref{remark1.4-1.6-w}.

  %in  the sense of Definition \ref{definition1.2-1-6-w} (where $t_0=-\tau$).

 \item[$(ii)$] The condition $(ii)$ is proposed for $\phi$ and $y_0$.
 In fact,
 $(1)$ in $(ii)$ of Theorem \ref{20240406-yubiao-Theorem-SingularitiesOfHistoricalValueAtEndpoints} provides a compatibility condition of $\phi$ at time $-\tau$ and $0$, while $(2)$ and $(3)$ in $(ii)$ of Theorem \ref{20240406-yubiao-Theorem-SingularitiesOfHistoricalValueAtEndpoints} give compatibility conditions of $\phi$ and $y_0$ at time $0$.  Theorem \ref{20240406-yubiao-Theorem-SingularitiesOfHistoricalValueAtEndpoints}  shows how the singularities of $\phi$ in $t_0=-\tau$ and $t_0=0$, and of $y_0$ collectively interact and influence those of
    $y(\cdot, \cdot;y_0,\phi)$ in  the periodic time sequence $\{j\tau\}_{j\in \mathbb N}$.

 \item[$(iii)$]
The time sequence $\{j\tau\}_{j\in \mathbb N}$ in  Theorem \ref{20240406-yubiao-Theorem-SingularitiesOfHistoricalValueAtEndpoints}
corresponds to the situation excluded by Theorem \ref{20230510-yb-theorem-PropagationOfSingularitiesForGeneralSolutions}:
    $t_0=-\tau$ and $t_0=0$.
    %The propagation of singularities for the solution $y(\cdot;y_0,\phi)$ starts from $t=0$ in Theorem \ref{20240406-yubiao-Theorem-SingularitiesOfHistoricalValueAtEndpoints},  while that in Theorem \ref{20230510-yb-theorem-PropagationOfSingularitiesForGeneralSolutions} starts from $t_0\in(-\tau,0)$. The reason is that when $\phi\in L^2((-\tau,0)\times\Omega)$,  it may have no value at $t=-\tau$, but $y(\cdot;y_0,\phi)$ has the value at $t=0$.

     \item[$(iv)$]   The propagation of sigularities that appeared in the above three main theorems presents some kind of hyperbolic effect in equation \eqref{20220510-yb-OriginalSystemWithTimeDelay}. Such an effect is not expected for classical heat equations because of the smoothing effect. So, these theorems show huge differences between heat equations with and without time delay.

    \item[$(v)$]  The main theorems can be extended to the following situations with small modifications of the methods developed in this paper:
    \begin{itemize}
        \item any bounded domain $\Omega$, where equation \eqref{20220510-yb-OriginalSystemWithTimeDelay} is verified, without assumptions on the boundary;
        \item other boundary conditions, such as the (homogeneous or nonhomogeneous) Neumann or Robin boundary condition;
        \item and other time-independent (possibly high-order) elliptic differential operators (replacing the Laplacian operator in equation \eqref{20220510-yb-OriginalSystemWithTimeDelay}).

    \end{itemize}

\end{enumerate}

\end{remark}

{\bf Strategy to prove main theorems} Our strategy is as follows:

\begin{itemize}[leftmargin=4em]
    \item[(s1)] We obtained an explicit expression of the solutions of equation \eqref{20220510-yb-OriginalSystemWithTimeDelay} in terms of the classical heat semigroup (see
    Theorem \ref {20230510-yb-proposition-ExpressionForEvolutionWithTimeDelay}). Although it might be possible to derive this expression  indirectly from \cite[Theorem 4.2 and Example 1, pp. 365-366]{Nakagiri-1981},
we obtained it directly using a method developed in  \cite{Wang-Zhang-Zuazua-2022}.
More importantly, we observed the propagation of singularities  through the aforementioned expression.

 \item[(s2)] To prove the propagation of singularities, using the expression of the solutions mentioned above, we derived some local properties of the solutions of the heat equation, including its local relationship with external forces (see Proposition \ref{2021218-yb-lemma-NonhomogenuousHeat} and Lemma \ref{2021220-yb-lemma-IntialTimeCaseForNonhomogenuousHeat}) and its local relationship with initial data (see Proposition \ref{20231210-yb-proposition-ReverseRegularityForSolutionsOFHeatEquation} and Lemma \ref{2021220-yb-lemma-IntialTimeCaseForNonhomogenuousHeat}).
  \end{itemize}

{\bf Novelty} The novelties of this paper can be summarized as follows:
\begin{itemize}[leftmargin=4em]
    \item[(n1)] The phenomenon of propagation of singularities that appeared in the main theorems seems to have been discovered for the first time in the literature.

    \item[(n2)] The $\tau$-periodical jumping propagation of singularities is a novel.
      The periodicity phenomenon for heat equations with time delay is demonstrated from the point of view of propagation of singularities. This viewpoint just provides the first glance at this phenomenon and the possibilities in further study on it, which may enhance the understanding of equation \eqref{20220510-yb-OriginalSystemWithTimeDelay}.
     \item[(n3)] We observed the above phenomena from the expression mentioned in $(s1)$. Such new observations can be explained in another way:

        \item[(n4)]  It is revealed (in Theorem \ref{20240406-yubiao-Theorem-SingularitiesOfHistoricalValueAtEndpoints})  that the initial data and the historical values at $t=-\tau$ and $t=0$ can interact and produce collective influences on the propagation of singularities of solutions for heat equations with time delay.
\end{itemize}

% The main ideas in the proofs of Theorems \ref{20230510-yb-theorem-PropagationOfSingularitiesOnlyWithInitialdata} and \ref{20230510-yb-theorem-PropagationOfSingularitiesForGeneralSolutions} are stated as follows: we express the solutions of equation \eqref{20220510-yb-OriginalSystemWithTimeDelay} in terms of the classical heat semigroup (see Section \ref{section-ExplicitExpression}), and then study the desired properties for this semigroup (see Section \ref{section-LocalRegularity}) to develop techniques for the analysis of equation \eqref{20220510-yb-OriginalSystemWithTimeDelay}.

\section{Local regularity properties for the heat equation}
\label{section-LocalRegularity}

This section aims mainly to present the following properties about heat equations:
(i)  the connection between the local regularities of the initial data and the corresponding solutions of heat equations (see Proposition \ref{20231210-yb-proposition-ReverseRegularityForSolutionsOFHeatEquation});
(ii) the connection between the local regularities of nonhomogeneous terms and the corresponding solutions of the heat equations (see Proposition \ref{2021218-yb-lemma-NonhomogenuousHeat}).

We start with the following definitions: denote by $\Delta_D$ the Laplacian operator with , i.e., the following linear unbounded operator over $L^2(\Omega)$:
\begin{align}\label{20241224-yb-LaplacianWithDiricihletBoundaryCondition}
    \Delta_D f :=  \Delta f,~   f \in H^2(\Omega) \cap H_0^1(\Omega).
\end{align}
Write $\{\lambda_i\}_{i\in\mathbb{N}^+}\subset \mathbb{R}^+$ and $\{e_i\}_{i\in\mathbb{N}^+}$ for the eigenvalues and the corresponding eigenfunctions, normalized in $L^2(\Omega)$,
    of the operator $-\Delta_D$.
Define the following Hilbert space with $s\in \mathbb{R}$:
\begin{eqnarray}\label{def-space-with-boundary-condition}
    \mathcal H^s :=
    \left\{f= \sum_{i=1}^{+\infty}a_i e_i : \{a_i\}_{i\in\mathbb{N}^+}\subset \mathbb R,\;
    \sum_{i=1}^{+\infty}|a_i|^2 \lambda_i^{s}<+\infty
    \right\},
    %(=D(\Delta^{\frac{s}{2}})),
\end{eqnarray}
   with the inner product:
\begin{align*}
    \langle f, g\rangle_{\mathcal H^s}:=
    \sum_{i=1}^{+\infty} a_i b_i \lambda_i^{s}
    ~\text{when}~
    f= \sum_{i=1}^{+\infty} a_i e_i \in \mathcal H^s
    ~\text{and}~
    g= \sum_{i=1}^{+\infty} b_i e_i \in \mathcal H^s.
\end{align*}
% Let
% \begin{align*}
%     \mathcal{H}^{-\infty}:=\cup_{s\in\mathbb{R}}\mathcal{H}^s.
% \end{align*}
Write $\{ e^{t \Delta_D }  \}_{t\geq 0}$ for the heat semigroup, that is, the  $C_0$-semigroup on $L^2(\Omega)$, generated by $\Delta_D$.
%  Then
%      \begin{equation}\label{yu-10-17-1}
%     e^{t \Delta}  f  :=   \sum_{i=1}^{+\infty}
%             e^{-\lambda_i t}a_ie_i,  ~t\geq 0,
%     ~\mbox{for each}~  f = \sum_{i=1}^{+\infty} a_i e_i.
% \end{equation}
For each $s\in \mathbb{R}$,
          $\{ e^{t \Delta_D}  \}_{t\geq 0}$ is also the $C_0$-semigroup on $\mathcal{H}^s$, with the generator
$\Delta_D$ and its domain $\mathcal{H}^{s+2}$.

The main results in this section are presented below.
The first one concerns the connection between the local regularities of nonhomogeneous terms and the corresponding solutions of the heat equations (which will be used in the proof of Theorem \ref{20230510-yb-theorem-PropagationOfSingularitiesForGeneralSolutions}).

\begin{proposition}\label{2021218-yb-lemma-NonhomogenuousHeat}
Let $( t_0,x_0 ) \in  \mathbb R \times \mathbb R^n$. Let $W$ be an open neighborhood of $(t_0,x_0)$. Suppose that  two distributions
    $y,f \in \mathcal{D}'(W)$ satisfy
\begin{align}\label{20241218-yubiao-NonhomogenuousHeatEquation}
\partial_t y  - \Delta y = f ~\text{in}~ W.
% \left\{
%     \begin{array}{ll}
%          \partial_t y  - \Delta y = f &\text{in}~ (T_1,T_3) \times \Omega,\\
%           y = 0                       &\text{on}~ (T_1,T_3) \times \partial\Omega.
%     \end{array}
% \right.
\end{align}
Then for any $r,s\in \mathbb N$, the following statements are equivalent:
\begin{itemize}[leftmargin=4em]
    \item[$(i)$ ] $y \in H^{r+1,s}_{loc}(t_0,x_0) \cap H^{r,s+2}_{loc}(t_0,x_0)$.
    \item[$(ii)$]  $f\in H^{r,s}_{loc}(t_0,x_0)$.
\end{itemize}
%     Let $\{t_j\}_{j=1}^3 \subset \mathbb R$ be a strictly increasing sequence. Take two distributions
%     $y, f\in \mathcal{D}'((t_1,t_3) \times \Omega)$ such that the following equality
% \begin{align}\label{20241218-yubiao-NonhomogenuousHeatEquation}
% \partial_t y  - \Delta y = f ~\text{in}~ (t_1,t_3) \times \Omega
% % \left\{
% %     \begin{array}{ll}
% %          \partial_t y  - \Delta y = f &\text{in}~ (T_1,T_3) \times \Omega,\\
% %           y = 0                       &\text{on}~ (T_1,T_3) \times \partial\Omega.
% %     \end{array}
% % \right.
% \end{align}
% holds in the distribution sense.
% Let $r,s\in \mathbb N$ and $x_0 \in \Omega$.
% Then, $y \in H^{r+1,s}_{loc}(t_2,x_0) \cap H^{r,s+2}_{loc}(t_2,x_0)$ if and only if $f\in H^{r,s}_{loc}(t_2,x_0)$.
\end{proposition}

The second main result in this section is stated in the following, which gives the relations between the local regularities of an initial data and corresponding solutions of the heat equations (which will be used in the proof of Theorem \ref{20230510-yb-theorem-PropagationOfSingularitiesOnlyWithInitialdata}).

\begin{proposition}\label{20231210-yb-proposition-ReverseRegularityForSolutionsOFHeatEquation}
%Let $y_0 \in \mathcal H^{-\infty}$ and $s\in \mathbb R$.
Let $T>0$ and $U$ be an open nonempty subset of $\mathbb R^n$.
% Define
%     \begin{equation}\label{20230524-yb-WeightedSolutionsOfHeatEquation}
%         f_{\alpha,\beta}(t,x;y_0) :=  \left[\partial_t^{\alpha} (t^{\beta} e^{t \Delta}y_0) \right](x),
%         ~(t,x)   \in  (0,+\infty)  \times  \Omega.
%     \end{equation}
% Then, it holds that
% $f_{\alpha,\beta}(\cdot;\cdot;y_0) \in L^2(0,T; H^{s+2(\beta-\alpha)+1}_{loc}( U ) )$
% if and only if $y_0 \in H^s_{loc}( U )$.
Suppose that $y\in C([0,T); \mathcal{D}'(U) )$ satisfies
\begin{align}\label{20241223-yb-HeatEquationForRegularityInvolvingIntialData}
    \partial_t y  -  \Delta  y = 0
    ~\text{in}~
    (0,T) \times U.
%    ~~ y|_{t=0} = y_0  ~\text{in}~  U.
\end{align}
Then for any  $\alpha,\beta\in \mathbb N$ and $s\in \mathbb R$,
the following statements are equivalent:
\begin{itemize}[leftmargin=4em]
    \item[$(i)$] The function $t\rightarrow \partial_t^{\alpha} (t^{\beta} y) $ ($t\in (0,T)$) belongs to the space $L^2((0,T); H^{s+2(\beta-\alpha)+1}_{loc}( U ) )$.
    \item[$(ii)$]  The function $y|_{t=0}$ belongs to the space $H^s_{loc}( U )$.
\end{itemize}
%
    %    with the following estimate
    %    \begin{align}
        %       C \|y_0\|_{H_A^s( U_2 )}
        %        \leq  \| t^j e^{-tA} y_0 \|_{ L^2(0,T; H^{s+2j+1}( U_1 ) ) }
        %          +  \| y_0 \|_{L^2(\Omega)},
        %    \end{align}
    %where $C>0$ is independent of $y_0$.
\end{proposition}

The following result is a consequence of Proposition \ref{20231210-yb-proposition-ReverseRegularityForSolutionsOFHeatEquation}.
%which will not be used in the proofs of main theorems in this paper but can enhance the understanding on them.

\begin{corollary}\label{20240401-yubiao-Corollay-ComparableAtInitialTimeForHeatSolutions}
    Let $U\subset \Omega$ be an open and nonempty subset, and let $\beta\in \mathbb N$.
     Suppose that $y_0 \in \cup_{k\in\mathbb R}\mathcal{H}^k$ satisfies $y_0 \in H^s_{loc}( U )$ for some $s\in \mathbb R$.   Let
    \begin{equation}\label{20240401-yb-WeightedSolutionsOfHeatEquation}
        f_{\beta}(t,x) :=  \chi_{(0,+\infty)} (t) t^{\beta} (e^{t^+ \Delta_D }y_0)(x),
        ~(t,x)   \in  \mathbb R  \times  \Omega.
    \end{equation}
 Then,  for each nonnegative integer $\alpha$ with $\alpha \leq \beta$,
 \begin{align}\label{20240402-yubiao-ContinuityAtInitialTimeForHeatSolution}
     f_{\beta}\in  H^{\alpha}(\mathbb R; H^{s+1 + 2(\beta-\alpha)}_{loc} (U) )
     \cap C^{\infty}( (\mathbb R \setminus\{0\}) \times \Omega ).
     % ~\text{for each integer}~
     % \alpha \leq k.
\end{align}
\end{corollary}

\begin{remark}
    The index $\alpha$ in Corollary \ref{20240401-yubiao-Corollay-ComparableAtInitialTimeForHeatSolutions} cannot be greater than $\beta$. The reason is as follows: for general initial data $y_0$, the left and right limits of the $\beta$-th time derivative of $f_{\beta}$ (in \eqref{20240401-yb-WeightedSolutionsOfHeatEquation}) at $t=0$ cannot be patched continuously.
\end{remark}

Next, we prove the above results. We start by proving Proposition \ref{2021218-yb-lemma-NonhomogenuousHeat}.

\begin{proof}[Proof of Proposition \ref{2021218-yb-lemma-NonhomogenuousHeat}]
We arbitrarily fix $r,s\in \mathbb N$.
    First, we show $(i)\Rightarrow (ii)$.  Suppose that $(i)$ holds.
%\begin{align*}
  %  y  \in H^{r+1,s}_{loc}(t_0,x_0) \cap H^{r,s+2}_{loc}(t_0,x_0).
%\end{align*}
Then, there is $\varepsilon>0$ and an open neighborhood $U_1$  of $x_0$ such that  $(t_0 - \varepsilon_1, t_0 + \varepsilon_1) \times U_1\subset W$ and
\begin{align*}
    y \in H^{r+1,s}( (t_0 - \varepsilon_1, t_0 + \varepsilon_1) \times U_1)  )
    \cap H^{r,s+2}( (t_0 - \varepsilon_1, t_0 + \varepsilon_1) \times U_1 ) ).
\end{align*}
   This implies
\begin{align*}
   \partial_t y, \Delta y \in H^{r,s}( (t_0 - \varepsilon_1, t_0 + \varepsilon_1) \times U_1).
\end{align*}
Since $(y,f)$ satisfies equation \eqref{20241218-yubiao-NonhomogenuousHeatEquation} over $(t_0 - \varepsilon_1, t_0 + \varepsilon_1) \times U_1$,
the above yields
\begin{align*}
   f = \partial_t y - \Delta y \in H^{r,s}( (t_0 - \varepsilon_1, t_0 + \varepsilon_1) \times U_1).
\end{align*}
Then, from the definition of $H^{r,s}_{loc}(t_0,x_0)$ in \eqref{yu-10-15-2} it follows that $f \in H^{r,s}_{loc}(t_0,x_0)$, that is, $(ii)$ is true.

Next, we show $(ii)\Rightarrow (i)$. We assume that $(ii)$
holds.
%\begin{align}\label{20241219-yb-SufficiencyForNonhomogeneousCase}
%    f \in H^{r,s}_{loc}(t_0,x_0).
%\end{align}
By  the definition of $H^{r,s}_{loc}(t_0,x_0)$ in \eqref{yu-10-15-2},
we see that to prove $(i)$, it  suffices to show that for each $\alpha \in \mathbb N$ and $\beta \in \mathbb N^n$ with $\alpha \leq r$ and $|\beta| \leq s$,
\begin{align}\label{20241219-yb-ReducedVersionInSufficiencyForNonhomogeneousCase}
    y_{\alpha,\beta} := \partial_t^{\alpha} \partial_x^{\beta} y  \in H^{1,0}_{loc}(t_0,x_0) \cap H^{0,2}_{loc}(t_0,x_0).
\end{align}
For this purpose, we arbitrarily fix $\alpha \in \mathbb N$ and
$\beta \in \mathbb N^n$ such that $\alpha \leq r$ and $|\beta| \leq s$.
By $(ii)$,  there is an  open neighborhood $U_2$ (of $x_0$), with a smooth boundary $\partial U_2$,  and  $\varepsilon_2 > 0$ such that
\begin{align*}
(t_0 - \varepsilon_2, t_0 + \varepsilon_2) \times U_2  \subset W
~\text{and}~
    f \in H^{r,s}( (t_0 - \varepsilon_2, t_0 + \varepsilon_2) \times U_2) ).
\end{align*}
This implies
\begin{align}\label{20241219-yb-ReducedNonhomogeneousTerm}
    \partial_t^{\alpha} \partial_x^{\beta} f
    \in  L^2(
        (t_0 - \varepsilon_2, t_0 + \varepsilon_2) \times U_2)
    ).
\end{align}
In addition, since $(y,f)$ satisfies \eqref{20241218-yubiao-NonhomogenuousHeatEquation} and $y_{\alpha,\beta} = \partial_t^{\alpha} \partial_x^{\beta} y$ (see \eqref{20241219-yb-ReducedVersionInSufficiencyForNonhomogeneousCase}),
we have
\begin{align}\label{20241219-yb-ReducedHeatEquationWithNonhomogeneousTerm}
    \partial_t y_{\alpha,\beta}   -  \Delta y_{\alpha,\beta}
    =  \partial_t^{\alpha} \partial_x^{\beta} f
    ~\text{over}~
     (t_0 - \varepsilon_2, t_0 + \varepsilon_2) \times U_2.
\end{align}

Now,  we will use \eqref{20241219-yb-ReducedHeatEquationWithNonhomogeneousTerm} and \eqref{20241219-yb-ReducedNonhomogeneousTerm} to prove
\begin{align}\label{20241219-yb-RegularityOfReducedSolution}
    y_{\alpha,\beta}
    \in  H^1_{loc}( (t_0 - \varepsilon_2, t_0 + \varepsilon_2);  L^2_{loc}(U_2) )
    \cap   L^2_{loc}((t_0 - \varepsilon_2, t_0 + \varepsilon_2);  H^2_{loc}(U_2)).
\end{align}
Once this is  done, \eqref{20241219-yb-ReducedVersionInSufficiencyForNonhomogeneousCase}  follows at once and
the proof of $(ii)\Rightarrow (i)$ is finished.

The remainder is to show \eqref{20241219-yb-RegularityOfReducedSolution}. We arbitrarily take a subdomain $V \Subset U_2$ and a
cut-off function $\rho \in C_0^{\infty}(U_2)$ with $\rho \equiv 1$ over $V$. Note that the boundary $\partial U_2$ is smooth and bounded.
Then, by \eqref{20241219-yb-ReducedNonhomogeneousTerm} and \cite[Theorem 5, Page 382]{Evans-2010-AMS}, there is a function
\begin{align}\label{20241220-yb-ConstructedHeatSolution}
 \hat y
    \in  H^1( (t_0 - \varepsilon_2, t_0 + \varepsilon_2);  L^2(U_2) )
    \cap   L^2((t_0 - \varepsilon_2, t_0 + \varepsilon_2);  H^2(U_2))
\end{align}
such that
\begin{align*}
\left\{
    \begin{array}{ll}
     \partial_t \hat y  -  \Delta  \hat  y  =  \rho \partial_t^{\alpha} \partial_x^{\beta} f
     &~\text{in}~
     (t_0 - \varepsilon_2, t_0 + \varepsilon_2) \times U_2,
     \\
     \hat y  = 0  &~\text{on}~
     (t_0 - \varepsilon_2, t_0 + \varepsilon_2) \times \partial U_2,
     \\
     \hat y|_{t=t_0 - \varepsilon_2}  = 0  &~\text{in}~  U_2.
    \end{array}
\right.
\end{align*}
Since $\rho \equiv 1$ over $V$, the above, combined with \eqref{20241219-yb-ReducedHeatEquationWithNonhomogeneousTerm},  implies
\begin{align*}
    \partial_t (y_{\alpha,\beta} - \hat y)   - \Delta (y_{\alpha,\beta} - \hat y) = 0
    ~\text{over}~
    (t_0 - \varepsilon_2, t_0 + \varepsilon_2) \times V.
\end{align*}
At the same time, since the heat operator $\partial_t  - \Delta $ is hypoelliptic (see \cite[Definition 11.1.2 and (iii) and (iv) of Theorem 11.1.1]{Hormander-2}),  we find from the last equality that
\begin{align*}
    y_{\alpha,\beta} - \hat y   \in  C^{\infty} ( (t_0 - \varepsilon_2, t_0 + \varepsilon_2) \times V).
\end{align*}
Because $V\Subset U_2$ was arbitrarily taken, the above, along with \eqref{20241220-yb-ConstructedHeatSolution}, leads to \eqref{20241219-yb-RegularityOfReducedSolution}. This
completes the proof of Proposition \ref{2021218-yb-lemma-NonhomogenuousHeat}.
\end{proof}

%\subsection{Proof of Proposition \ref{20231210-yb-proposition-ReverseRegularityForSolutionsOFHeatEquation}}

To prove Proposition \ref{20231210-yb-proposition-ReverseRegularityForSolutionsOFHeatEquation}, we need the following two lemmas. The first lemma is about a global regularity estimate for the heat semigroup.

\begin{lemma}\label{20231210-yb-lemma-IdentityForWeightedHeatSolutions}
Let $y_0 \in \cup_{k\in\mathbb R}\mathcal{H}^k$ and $s\in \mathbb R$. Let $\alpha,\beta\in \mathbb N$. Then the following two statements are equivalent:
\begin{itemize}[leftmargin=4em]
    \item[$(i)$]  The function
    $t\mapsto \partial_t^{\alpha} (t^{\beta} e^{t \Delta_D} y_0)$ ($t\geq 0$) is in the space $L^2(\mathbb{R}^+; \mathcal H^{s+2(\beta-\alpha)+1})$.

    \item[$(ii)$]   $y_0\in \mathcal{H}^{s}$.
\end{itemize}
Moreover, if either $(i)$ or $(ii)$ holds, then
\begin{equation}\label{20231210-yb-EstimateOnSolutionsOfHeatEquation}
    \int_{\mathbb{R}^+}  \| \partial_t^{\alpha} (t^{\beta} e^{t \Delta_D} y_0)\|_{\mathcal H^{s+2(\beta-\alpha)+1}}^2
    dt
=  \left(
    \int_{\mathbb{R}^+} | \partial_t^{\alpha} (t^{\beta} e^{-t}) |^2 dt
    \right)
    \|y_0\|_{\mathcal H^s}^2.
\end{equation}
\end{lemma}

% \begin{lemma}\label{yu-lemma-9-19-1}
% % Let $T>0$ and  $f\in L^2((0,T) \times \Omega)$. Let $r,s\in\mathbb{N}$.  Then, the solution $y_f$ to the following equation
% % \begin{align}\label{20240502-yubiao-NonhomogenuousHeatEquation}
% % \left\{
% %     \begin{array}{ll}
% %          \partial_t y  - \Delta y = f &\text{in}~ (0,T) \times \Omega,\\
% %           y = 0                       &\text{on}~ (0,T) \times \partial\Omega,\\
% %           y(0,\cdot) = 0              &\text{in}~ \Omega
% %     \end{array}
% % \right.
% % \end{align}
% % satisfies that $y_f \in H^{r+1}((0,T); \mathcal H^s)
% % \cap
% % H^r((0,T);\mathcal H^{s+2})$
% % if and only if $f\in H^r((0,T);\mathcal H^s)$ and
% % \begin{align}
% %     \partial_t^p f(0,\cdot) =0  ~\text{in}~ \mathcal H^s
% %     ~\text{when}~
% %     p \in \mathbb N
% %     ~\text{and}~  p \leq r-2.
% % \end{align}
%     Suppose that
%     $f(\cdot)\in H^r(0,T;\mathcal{H}^s)$. Let $y_f$ be
%     the solution of the equation $\partial_ty=\Delta y+f$ in $(0,T)\times\Omega$
%     with $y|_{\partial\Omega}=$ and $y|_{t=0}=0$.
% Then  $y_f$ has the following properties:
% \begin{enumerate}
%   \item [$(i)$] If $r=0$, then $y_f\in H^1(0,T;\mathcal{H}^s)\cap L^2(0,T;\mathcal{H}^{s+2})$.
%   \item [$(ii)$] If there is $\varepsilon>0$   such that
%     $\chi_{[0,\varepsilon]}(\cdot)f(\cdot)=0$ over $[0,T]$, then
%     $y_f\in H^{r+1}(0,T;\mathcal{H}^s)\cap  H^r(0,T;\mathcal{H}^{s+2})$.
% \end{enumerate}
% \end{lemma}

\begin{proof}
First of all, since $y_0\in \cup_{k\in\mathbb R}\mathcal{H}^k$, it is clear that $y_0 \in \mathcal H^m$ for some $m\in \mathbb R$.
Thus,  we can write $y_0 = \sum_{i\geq1} y_{0,i} e_i$, with
$\big\{ \lambda_i^{ m/2 } y_{0,i} \big\}_{i\geq 1} \in l^2$. Then we have
$$
    \partial_t^{\alpha} (t^{\beta} e^{t \Delta_D} y_0)
      = \sum_{i\geq 1}  \partial_t^{\alpha} (t^{\beta} e^{ -\lambda_i t } )
       y_{0,i}  e_i,
       ~ t\geq 0.
$$

Next, we aim to show that $(i) \Rightarrow (ii)$. Suppose that $(i)$ is true. Then we have
\begin{align*}
        \int_{\mathbb{R}^+}  \| \partial_t^{\alpha} (t^{\beta} e^{t \Delta_D} y_0)\|_{\mathcal H^{s+2(\beta-\alpha)+1}}^2
        dt
        =& \sum_{i\geq 1}   \left(
            \int_{\mathbb{R}^+}
         | \partial_t^{\alpha} (t^{\beta} e^{- \lambda_i t}) |^2 dt
        \right)
        y_{0,i}^2 \lambda_i^{s+2(\beta-\alpha)+1}
        <  + \infty.
\end{align*}
Meanwhile, one can directly check that for each $\lambda>0$,
\begin{align*}
        \int_{\mathbb{R}^+}
        | \partial_t^{\alpha} (t^{\beta} e^{- \lambda t}) |^2 dt
      = \lambda^{-2(\beta-\alpha) -1}
         \int_{\mathbb{R}^+} | \partial_t^{\alpha} (t^{\beta} e^{-t}) |^2 dt.
\end{align*}
The above two equalities yield
\begin{align*}
\int_{\mathbb{R}^+}
        | \partial_t^{\alpha} (t^{\beta} e^{- t}) |^2 dt
        \sum_{i\geq 1} y_{0,i}^2 \lambda_i^{s}
        =
        \int_{\mathbb{R}^+}  \| \partial_t^{\alpha} (t^{\beta} e^{t \Delta_D} y_0)\|_{\mathcal H^{s+2(\beta-\alpha)+1}}^2
        dt
        < + \infty,
    \end{align*}
    which leads to  \eqref{20231210-yb-EstimateOnSolutionsOfHeatEquation} and
    shows  $(ii)$.

    The proof of $(ii)\Rightarrow (i)$ can be done in a similar way.  This completes the proof.
\end{proof}

% The following result concerns the local smoothing effect of the initial data to the solution
% of the classical heat equation. We omit its proof in this paper since it is similar to that
% of \cite[Lemma 5.3]{WZZ-2021}.

% \begin{lemma}\label{20230524-yb-lemma-PropagationOfSmoothnessForHeatSolutions}
%     Let $U$ be a nonempty open subset of $\Omega$. Let $z \in \mathcal{H}^{-\infty}$
%     satisfy that $z|_U\in C^\infty(U)$.     Then, for each  $\rho\in C_0^\infty(U)$,  the function
%     $$
%     (t,x)  \mapsto \rho(x) (e^{t\Delta}z)(x), ~(t,x)\in [0,+\infty)\times \Omega
%     $$
%     belongs to     $C^{\infty}([0,+\infty)\times \Omega)$.
% \end{lemma}

The second lemma mentioned above concerns solutions of local infinite differentiability for the classical heat equation. It is a modified version of \cite[Theorem 7, Page 388]{Evans-2010-AMS} and can be proven by small modifications. The proof is put in the Appendix for the sake of the completeness.

\begin{lemma}\label{20230524-yb-lemma-PropagationOfSmoothnessForHeatSolutions}
Let $T>0$ and $U$ be an open nonempty subset of $\mathbb R^n$. Let $y_0 \in  C^{\infty}(U)$ and $f\in C^{\infty}([0,T) \times U)$.
Suppose that  $y \in C([0,T); \mathcal{D}'(U) )$ satisfies
\begin{align}\label{20241220-yb-InitialDataAndNonhomogeneousTerm}
    \partial_t y  - \Delta y  =  f  ~\text{in}~  (0,T) \times U;
    ~~ y|_{t=0} = y_0  ~\text{in}~  U.
\end{align}
Then, we have $y \in C^{\infty}([0,T) \times U)$.
\end{lemma}

Now we are in a position to prove Proposition \ref{20231210-yb-proposition-ReverseRegularityForSolutionsOFHeatEquation}, based on Lemmas \ref{20231210-yb-lemma-IdentityForWeightedHeatSolutions} and \ref{20230524-yb-lemma-PropagationOfSmoothnessForHeatSolutions}.

\begin{proof}[Proof of Proposition \ref{20231210-yb-proposition-ReverseRegularityForSolutionsOFHeatEquation}]
 We arbitrarily fix $\alpha,\beta\in \mathbb N$ and $s\in \mathbb R$.
Without loss of generality, we assume
\begin{align}\label{20241223-yb-AssumptionWithoutLossOfGenerality}
    U \subset \Omega.
\end{align}
Otherwise, we can shrink $U$ and move it into $\Omega$ by translations; this will not
harm the conclusion to be proven because it is about local regularities. The reason for assuming \eqref{20241223-yb-AssumptionWithoutLossOfGenerality} is to use Lemma \ref{20231210-yb-lemma-IdentityForWeightedHeatSolutions} in a convenient way.

Write $y_0 := y|_{t=0}$.  The proof is divided into the following two steps.

\vskip 5pt
\noindent\emph{
Step 1. We prove $(ii)\Rightarrow (i)$.
}

Suppose that $y_0 \in H^s_{loc}( U )$.  We arbitrarily take an open nonempty subset $V_1 \Subset U$. Let $\rho_1 \in C_0^\infty(U)$ be a cutoff function  such that
    $\rho_1  = 1$  over $V_1$.
Since $y_0 \in H^s_{loc}( U )$ and $U \subset \Omega$ (see \eqref{20241223-yb-AssumptionWithoutLossOfGenerality}), we see from Lemma \ref{20241230-yb-lemma-InterEquivalenceBetweenTwoSobolevSpaces} in the Appendix that $\rho_1 y_0 \in  \mathcal{H}^s$ (when extended to $\Omega$).
    Then, according to Lemma  \ref{20231210-yb-lemma-IdentityForWeightedHeatSolutions},
    %the function $f_{\alpha,\beta}$ defined by   \eqref{20230524-yb-WeightedSolutionsOfHeatEquation} satisfies
    \begin{equation}\label{20231209-yb-SobolevRegularityOfWeightedSolution}
        \partial_t^{\alpha}[ t^{\beta} e^{t\Delta_D} (\rho_1 y_0) ]  \in  L^2( (0,+\infty);  \mathcal H^{s+2(\beta-\alpha)+1}).
        %f_{\alpha,\beta}(\cdot,\cdot; \rho_1 y_0) \in L^2(0,+\infty;  \mathcal H^{s+2(\beta-\alpha)+1}).
    \end{equation}
Meanwhile, since $(1-\rho_1)y_0 =0 $ on $V_1$, it follows from \eqref{20241223-yb-HeatEquationForRegularityInvolvingIntialData} that the function $y - e^{t\Delta_D} (\rho_1 y_0)$ satisfies the following equation
\begin{align*}
    \partial_t z   - \Delta  z  = 0   ~\text{in}~   (0,T)  \times V_1;
    ~~z|_{t=0} =  (1-\rho_1)y_0 =0   ~\text{in}~ V_1.
\end{align*}
Thus, we can apply Lemma \ref{20230524-yb-lemma-PropagationOfSmoothnessForHeatSolutions} (with $(U, y_0, f)$ replaced by $(V_1, 0, 0)$) to obtain
    \begin{equation*}
        y - e^{t\Delta_D} (\rho_1 y_0)
        %f_{\alpha,\beta}(\cdot,\cdot; (1-\rho_1) y_0)
     \in C^{\infty} ( [0,+\infty) \times V_1 ).
    \end{equation*}
    This, together with \eqref{20231209-yb-SobolevRegularityOfWeightedSolution}, yields
    \begin{equation*}
        \partial_t^{\alpha}( t^{\beta} y )   =
        \partial_t^{\alpha}[ t^{\beta} e^{t\Delta_D} (\rho_1 y_0) ]
        +   \partial_t^{\alpha}[ t^{\beta} (y - e^{t\Delta_D} (\rho_1 y_0)) ]
          % f_{\alpha,\beta}(\cdot,\cdot;  y_0)   =  f_{\alpha,\beta}(\cdot,\cdot; \rho_1 y_0)       +
          %    f_{\alpha,\beta}(\cdot,\cdot; (1-\rho_1) y_0)
      \in L^2( (0,T) ; H^{s+2(\beta-\alpha)+1}_{loc}(V_1) ).
    \end{equation*}
    Because $V_1 \Subset U$ was arbitrarily chosen, the above implies
    $\partial_t^{\alpha}( t^{\beta} y )
      \in L^2( (0,T); H^{s+2(\beta-\alpha)+1}_{loc}( U ) )$,
    i.e., $(i)$ is true.

\vskip 5pt
\noindent\emph{
    Step 2. We prove $(i)\Rightarrow (ii)$.
}

   %Let $y_0\in L^2(\Omega)$.
   We assume that $(i)$ holds.
   %    \begin{align}\label{20230520-yb-AssumptionForNecessity}
    %         \partial_t^\alpha(t^{\beta} y)
     %   \in L^2( (0,T); H^{s+2(\beta-\alpha)+1}_{loc}( U ) ),
    %\end{align}
  %   as a function of $t$.
We will prove $(ii)$ by induction.
%Without loss of generality, we suppose that $s>0$.
To this end, we arbitrarily take a subdomain $U_1 \Subset U$.
Since $y \in C([0,T); \mathcal{D}'(U))$, we have $y_0 \in \mathcal{D}'(U)$.
Then, one can directly check that (see Subsection \ref{20250203-DistributionsAndSobolevSpaces} in the Appendix)
\begin{align}\label{20250203-DistributionEmbeddedIntoSobolevSpaces}
    y_0 \in H^{s - m_0}_{loc}(U_1)
    ~~\text{for some}~~
    m_0 \in \mathbb N^+.
\end{align}
Let $m \in [1, m_0]$ be an integer. We show that if
\begin{align}\label{20231208-yb-SubOptimalRegularity-y0}
    y_0 \in H^{ s - m}_{loc}(U_1),
\end{align}
then
\begin{align}\label{2024023-yubiao-AimInInductionForNecessity}
    y_0 \in H^{s - m +1}_{loc}(U_1).
\end{align}
When this is done, one can conclude that $y_0 \in H^s_{loc}(U_1)$ by induction, which gives $y_0 \in H^s_{loc}(U)$ due to the arbitrariness of $U_1 \Subset U$.

    The rest is to show \eqref{2024023-yubiao-AimInInductionForNecessity}, based on \eqref{20231208-yb-SubOptimalRegularity-y0}.
    % For simplicity, we write
    % \begin{align}\label{20230520-yb-DefinitionOffForNecessityInMainProof}
    %     f(t):= e^{t \Delta} y_0,  ~t\geq 0.
    % \end{align}
    When  $k\in \{0,\ldots,\beta\}$, by \eqref{20231208-yb-SubOptimalRegularity-y0}, we apply the
    conclusion that $(ii)\Rightarrow (i)$  of this proposition (where $(s, \beta, U)$ is replaced by $(s-m,k, U_1)$)) to obtain
    \begin{align}\label{20231108-yb-RegularityForWeightedFuncitons}
        \partial_t^{\alpha} (t^k y)      \in  L^2( (0,T); H^{(s-m+1) +2(k-\alpha)}_{loc}( U_1 )),
    \end{align}
    where $ \partial_t^{\alpha} (t^k y)$ is viewed as  a function of $t$.

    Now, we arbitrarily take a cut-off function $\rho \in C_0^{\infty}(U_1)$.
    Note  that the function $\rho y$, when extended to $(0,T) \times \Omega$, solves the following heat equation with a non-homogeneous term:
    \begin{align}\label{20231210-yb-EquationInMainProof}
        \left\{
        \begin{array}{ll}
            \partial_t (\rho y) - \Delta (\rho y) = g:= (-\Delta \rho ) y - 2 \nabla \rho \cdot \nabla y  &\text{in}  ~ (0,T) \times \Omega,\\
            (\rho y) =0  &\text{on}  ~(0,T)  \times \partial\Omega,\\
            (\rho y)|_{t=0}  = \rho  y_0   &\text{in}   ~ \Omega.
        \end{array}
        \right.
    \end{align}
    Since $\rho \in C_0^{\infty}(U_1)$ and $U_1 \subset U \subset \Omega$, by \eqref{20231108-yb-RegularityForWeightedFuncitons}, the
    term $g$ in \eqref{20231210-yb-EquationInMainProof}, which is viewed as a function of $t$, satisfies
    \begin{align}\label{20231208-yb-WeightedRegularityFor-g}
        \partial_t^{\alpha} (t^k g) \in L^2( (0,T); \mathcal H^{(s-m)+2(k-\alpha)})
        ~\text{for each}~
        k\in \{0,\ldots,\beta\}.
    \end{align}

    Next, we claim that the  solution $y_1$ to the following equation
    \begin{align}\label{20230519-yb-EquationWithZeroIntialDataInMainProof}
        \left\{
        \begin{array}{ll}
            \partial_t y_1 - \Delta y_1 = g  &\text{in}  ~ (0,T) \times \Omega,\\
            y_1 = 0  &\text{on}  ~(0,T)  \times \partial\Omega,\\
            y_1|_{t=0} = 0   &\text{in}   ~\Omega
        \end{array}
        \right.
    \end{align}
    satisfies that for each $k\in \{0,\ldots,\beta\}$,
    \begin{align}\label{20231208-yb-WeightedRegularityForHeatEquation}
         \text{the function}~
         t\mapsto \partial_t^{\alpha} (t^k y_1)  ~(t\geq0)
         ~\text{is in}~
         L^2( (0,T); \mathcal H^{(s-m+2)+2(k-\alpha)} ).
    \end{align}
  We prove the above claim by induction. First, it follows  from \eqref{20230519-yb-EquationWithZeroIntialDataInMainProof} and \eqref{20231208-yb-WeightedRegularityFor-g} (where $k=0$) that
  \begin{align*}
       \partial_t (\partial_t^{\alpha} y_1)
     - \Delta (\partial_t^{\alpha} y_1)
   = \partial_t^{\alpha} (\partial_t y_1  - \Delta y_1)
   =  \partial_t^{\alpha} g
   ~~\text{in}~   L^2( (0,T); \mathcal{H}^{(s-m)+2(0-\alpha)} ).
  \end{align*}
  This, together with Lemma \ref{20241223-yb-lemma-SolutionsWithNonhomogeneousTermInFunctionalSetting} (where $y$ and $s$ are replaced by $\partial_t^{\alpha} y_1$ and $(s-m)+2(0-\alpha)$, respectively) in the Appendix, leads to \eqref{20231208-yb-WeightedRegularityForHeatEquation} with  $k=0$.
   Next, we suppose that there is  a positive integer $N<\beta$ so that the statement \eqref{20231208-yb-WeightedRegularityForHeatEquation} holds for  all $k < N$.
    To prove    \eqref{20231208-yb-WeightedRegularityForHeatEquation} with $k=N$, we first obtain from \eqref{20230519-yb-EquationWithZeroIntialDataInMainProof} that
 \begin{align*}
     \partial_t \left[
            \partial_t^{\alpha} (t^N y_1)
       \right]
     - \Delta \left[
                  \partial_t^{\alpha} (t^N y_1)
              \right]
     = \partial_t^{\alpha} (t^N g)
           +  N \partial_t^{\alpha} (t^{N-1} y_1).
 \end{align*}
 Then, according to \eqref{20231208-yb-WeightedRegularityFor-g} (where $k=N$) and \eqref{20231208-yb-WeightedRegularityForHeatEquation} (with $k=N-1$),  we have
 \begin{align*}
     \partial_t \left[
        \partial_t^{\alpha} (t^N y_1)
    \right]
    -  \Delta \left[
        \partial_t^{\alpha} (t^N y_1)
    \right]
    \in  L^2( (0,T); \mathcal{H}^{(s-m)+2(N-\alpha)} ).
 \end{align*}
 %all the terms on right hand side of the above equality are in $L^2(0,T; \mathcal{H}^{(s-m)+2(N-\alpha)} )$.
 Therefore, by Lemma \ref{20241223-yb-lemma-SolutionsWithNonhomogeneousTermInFunctionalSetting} (where  $y$ and $s$ are replaced by $\partial_t^{\alpha} (t^N y_1)$ and $(s-m)+2(N-\alpha)$, respectively) in the Appendix, we get \eqref{20231208-yb-WeightedRegularityForHeatEquation} with $k=N$. Therefore,  \eqref{20231208-yb-WeightedRegularityForHeatEquation} is true by the induction.

%\eqref{20230520-yb-DefinitionOffForNecessityInMainProof}
    Now, because $\rho \in C_0^{\infty}(U_1)$  and $U_1 \subset U \subset \Omega$, it follows from  $(i)$ that
    \begin{align*}
         \partial_t^{\alpha} (t^{\beta} \rho y)
        \in L^2( (0,T); \mathcal{H}^{(s+1)+2(\beta-\alpha)}).
    \end{align*}
    Since $m\geq 1$, the above, together with \eqref{20231208-yb-WeightedRegularityForHeatEquation} (with $k=\beta$), implies
    \begin{align*}
        \partial_t^{\alpha} \left[
            t^{\beta} ( \rho y  - y_1 )
        \right]
        \in L^2( (0,T); \mathcal H^{(s-m+2)+2(\beta-\alpha)}),
        % \partial_t^{\alpha} \left[
        %         t^{\beta}  e^{t \Delta} (\rho y_0)
        %   \right]
        % = \partial_t^{\alpha} \left[
        % t^{\beta} ( \rho f  - y_1 )
        % \right]
        %     %\nonumber\\
        % = \rho \partial_t^{\alpha} (t^{\beta} f)
        %    -   \partial_t^{\alpha} (t^{\beta} y_1)
        % \in L^2(0,T; \mathcal H^{(s-m+2)+2(\beta-\alpha)}),
    \end{align*}
    Then,    we see from \eqref{20231210-yb-EquationInMainProof} and \eqref{20230519-yb-EquationWithZeroIntialDataInMainProof}  that the function $\rho y  - y_1$ satisfies the following heat equation
    \begin{align*}%\label{20230519-yb-EquationWithZeroIntialDataInMainProof}
        \left\{
        \begin{array}{ll}
            \partial_t z - \Delta z = 0  &\text{in}  ~ (0,T) \times \Omega,\\
            z = 0  &\text{on}  ~(0,T)  \times \partial\Omega,\\
            z|_{t=0}= \rho y_0  &\text{in}   ~\Omega.
        \end{array}
        \right.
    \end{align*}
    Thus, we conclude that
    \begin{align*}
        \partial_t^{\alpha} \left[
            t^{\beta}  e^{t \Delta_D } (\rho y_0)
        \right]
        =
        \partial_t^{\alpha} \left[
            t^{\beta} ( \rho y  - y_1 )
        \right]
        \in L^2( (0,T); \mathcal H^{(s-m+2)+2(\beta-\alpha)}).
    \end{align*}
    This, along with the smoothing effect and the decay property of the heat semigroup, implies
    \begin{align*}
        \partial_t^{\alpha} \left[
        t^{\beta}  e^{t \Delta_D} (\rho y_0)
        \right]
        \in L^2(\mathbb{R}^+; \mathcal H^{(s-m+2)+2(\beta-\alpha)}),
    \end{align*}
     From this and Lemma \ref{20231210-yb-lemma-IdentityForWeightedHeatSolutions}, we obtain
        $\rho  y_0 \in \mathcal H^{s-m+1}$.
        %\subset H^{s-m+1}(\Omega).
    Because $U_1 \subset U \subset \Omega$ and $\rho \in C_0^\infty(U_1)$ was arbitrarily taken, we see from Lemma \ref{20241230-yb-lemma-InterEquivalenceBetweenTwoSobolevSpaces} in the Appendix that $y_0 \in H^{s-m+1}_{loc}(U_1)$,
    i.e., \eqref{2024023-yubiao-AimInInductionForNecessity} is true.
    This completes the proof of Proposition \ref{20231210-yb-proposition-ReverseRegularityForSolutionsOFHeatEquation}.
    %    From \eqref{20230519-yb-PropertyOf-f}, the non-homogenous term in the above heat equation is in $L^2(0,\varepsilon; H^{s+2j})$.
    %
    %
    %
    %
    %   We apply Proposition \ref{20230510-yb-proposition-RegularityForSolutionsOFHeatEquation} (with $y_0$ replaced by $(1-\rho_3)y_0$) to obtain that when $t \geq 0$,
    %   \begin{align*}
        %       f(t) =& t^{j_0} e^{ - t A} (\rho_3 y_0)
        %       +   t^{j_0} e^{ - t A} ( (1-\rho_3) y_0)
        %       \nonumber\\
        %       \in&  t^{j_0} e^{ - t A} (\rho_3 y_0)  +  C^{\infty}([0,+\infty) \times \text{supp}\, \rho_3)
        %       .
        %   \end{align*}
    %   Then, from \eqref{20230519-yb-PropertyOf-f}, it follows that the function $t\mapsto t^{j_0} e^{ - t A} (\rho_3 y_0)$, $t\geq 0$ is in $L^2(0,\varepsilon; H^{s+2j_0+1}(\Omega) )$.
    %   Now, we apply Lemma \ref{2023-yb-lemma-RegularityForSolutionsOfHeatEquation}  (with $(y_0,j)$ replaced by $(\rho_3 y_0, j_0)$) to find that
    %   \begin{align*}
        %       \rho_3 y_0 \in \mathcal H^{s} = H^s(\Omega).
        %   \end{align*}
\end{proof}

Finally, we present the proof of Corollary \ref{20240401-yubiao-Corollay-ComparableAtInitialTimeForHeatSolutions}.

\begin{proof}[Proof of Corollary \ref{20240401-yubiao-Corollay-ComparableAtInitialTimeForHeatSolutions}]
    Let $\alpha \in \mathbb N \cap [0,\beta]$. We arbitrarily fix a nonnegative integer $k\leq \alpha$. Then we apply Proposition \ref{20231210-yb-proposition-ReverseRegularityForSolutionsOFHeatEquation} (with $\alpha$ replaced by $k$) and the decay property for the heat semigroup to obtain
\begin{align*}
     \partial_t^{k} f_{\beta}|_{ \{t>0\}}
    = \partial_t^k (t^{\beta} e^{t \Delta_D}y_0)
    %=     f_{k,\beta}
    \in &  L^2 ((0,+\infty); H^{s+1 + 2(\beta-k)}_{loc} (U) )
        \nonumber\\
    \subset& L^2 ((0,+\infty); H^{s+1 + 2(\beta - \alpha)}_{loc} (U) ).
\end{align*}
%where $f_{k,\beta}$ is given by \eqref{20230524-yb-WeightedSolutionsOfHeatEquation} (with $\alpha=k$).
This implies $f_{\beta}|_{\{ t>0\} }  \in H^{\alpha}
    \big(
        (0,+\infty); H^{s+1 + 2(\beta - \alpha)}_{loc} (U)
    \big)$.
Furthermore, because of the factor $t^{\beta}$ in $f_{\beta}$, one can check that when $k\in \mathbb N$ satisfies $k<\beta$,
\begin{align*}
    \lim_{t\rightarrow 0^+} \partial_t^k f_{\beta} (t) = 0
    ~\text{in}~    \cup_{m \in \mathbb R} \mathcal{H}^m.
\end{align*}
Meanwhile,  we note that $f_{\beta}|_{\{t<0\}} \equiv 0$. Thus, $f_{\beta}\in  H^{\alpha}(\mathbb R; H^{s+1 + 2(\beta-\alpha)}_{loc} (U) )$.
The function $f_{\beta}$ is also smooth outside $t=0$, due to the smoothing effect of the heat semigroup. Therefore,  \eqref{20240402-yubiao-ContinuityAtInitialTimeForHeatSolution} is true. This
finishes the proof of Corollary \ref{20240401-yubiao-Corollay-ComparableAtInitialTimeForHeatSolutions}.
\end{proof}

\section{Explicit  expression of solutions}
\label{section-ExplicitExpression}

 This section presents an explicit expression for solutions of  equation \eqref{20220510-yb-OriginalSystemWithTimeDelay} in terms of the semigroup $\{ e^{t \Delta_D }  \}_{t\geq 0}$. It plays an important role in the proofs of our main results.
Recall that $\Delta_D$ and $\mathcal{H}^s$ are introduced in \eqref{20241224-yb-LaplacianWithDiricihletBoundaryCondition} and \eqref{def-space-with-boundary-condition}, respectively.  The expression aforementioned is stated in the following theorem.

\begin{theorem}\label{20230510-yb-proposition-ExpressionForEvolutionWithTimeDelay}
Let $s \geq 0$. Let $y_0\in \mathcal{H}^s$ and $\phi\in L^2( (-\tau,0); \mathcal{H}^{s-1})$.
Then,  $y(t; y_0,\phi)|_{[0,+\infty)}$ belongs to the space $C([0,+\infty); \mathcal H^s)$,
%when restricted over $[0,+\infty) \times \Omega$,
and
satisfies
% \begin{equation}\label{zhang-8-28-1}
%     y(\cdot, \cdot; y_0,\phi) \in C([0,+\infty); \mathcal{H}^{s} )
%     \bigcap  L^2_{loc}([0,+\infty); \mathcal{H}^{s+1}).
% \end{equation}
% Furthermore, it holds that
\begin{equation}\label{zhang-00-8-23-3}
    y(t; y_0,\phi)  =  \varPhi(t) y_0  +a
    \int_{-\tau}^{ \min\{t-\tau,0\}  }
    \varPhi(t-\tau-\gamma)   \phi(\gamma)  d\gamma,
    ~t\geq 0,
\end{equation}
where
\begin{equation}\label{zhang-00-8-23-2}
    \varPhi(t) := e^{t\Delta_D}   +  \sum_{j=1}^{+\infty}
    \frac{a^j}{j!}  (t-j\tau)^j
    \chi_{[j\tau,+\infty)}(t)   e^{(t-j\tau)^+ \Delta_D},
    ~t\geq 0.
\end{equation}
%(here,  $r^+:=\max\{0,r\}$ for $r\in\mathbb{R}$).
\end{theorem}

\begin{remark}\label{20250118-yb-remark-ExpressionOfSolutions}
 $(i) $ Although it may be possible to derive the above expression indirectly from \cite[Theorem 4.2 and Example 1, pp. 365-366]{Nakagiri-1981},
we obtain it directly using a method developed in  \cite{Wang-Zhang-Zuazua-2022}.

$(ii)$ By \eqref{zhang-00-8-23-3}, we have that for each $y_0 \in L^2(\Omega)$,
$$
    y(t; y_0,0)  =  \varPhi(t) y_0,  ~\forall\, t\geq 0.
$$
Hence,  $\{ \varPhi(t) \}_{t\geq 0}$ is the flow generated by equation \eqref{20220510-yb-OriginalSystemWithTimeDelay} with $\phi=0$. Moreover,
 \eqref{zhang-00-8-23-2} provides a way to calculate this flow.

$(iii)$   The expression \eqref{zhang-00-8-23-2} gives the exact difference between the solution of the heat equations with and without time delay (i.e.,  $a=0$).  The sum on the right side of \eqref{zhang-00-8-23-2} is caused by the delay term $a y(t-\tau,x) $ in equation \eqref{20220510-yb-OriginalSystemWithTimeDelay}. This sum is independent of the historical value $\phi$.

$(iv)$ With $\tilde \phi$ the zero extension of $\phi$ over $\mathbb R$, \eqref{zhang-00-8-23-3} can be rewritten as follows:
\begin{align*}%\label{3.3-1-1-w}
    y(t; y_0,\phi) = \varPhi(t) y_0 +  a\int_{-\infty}^{t-\tau} \varPhi((t-\tau) - \gamma)  \tilde\phi(\gamma) d\gamma,
    ~ t \geq 0.
\end{align*}
This is quite similar to the classical variation-of-constants formula for ODEs.

$(v)$ Theorem \ref{20230510-yb-proposition-ExpressionForEvolutionWithTimeDelay} can be extended to be true for any $s \in \mathbb R$ by very small modifications of its proof.
\end{remark}

\begin{proof}[Proof of Theorem \ref{20230510-yb-proposition-ExpressionForEvolutionWithTimeDelay}]
We write
\begin{align*}
    (f*g)(t) :=  \int_0^t f(t-\eta) g(\eta) d\eta,~   t \geq 0\;\;\mbox{for any}\;\; f, g\in L^1_{loc}([0,+\infty); \mathcal{H}^{s-1}).
    \end{align*}
The proof is divided into the following steps.

\vskip 5pt
\noindent\textit{Step 1. To transform \eqref{20220510-yb-OriginalSystemWithTimeDelay} into an integral equation}

We start with writing a delay function in the convolution form. For this purpose, we introduce the following:
given $f\in L^1_{loc}([0,+\infty); \mathcal{H}^{s-1})$,
the function $\delta_{\tau} * f$ is also in  $L^1_{loc}([0,+\infty); \mathcal{H}^{s-1})$, and has the expression
    \begin{equation*}%\label{zhang-8-21-1}
        \delta_{\tau} * f (t)
        = \begin{cases}
            f(t-\tau), &t\geq \tau,\\
            0, &0\leq t < \tau.
        \end{cases}
    \end{equation*}
(The convolution here can also be understood in the distribution sense.)
    %  given $f\in C([0,+\infty); \mathcal{H}^p)$,
    %  with $p\in\mathbb{R}$, the function $\delta_{\tau} * f$ is piecewise continuous
    %  over $[0,+\infty)$ and has the expression
    % \begin{equation}\label{zhang-8-21-1}
    %     \delta_{\tau} * f (t)
    %     = \begin{cases}
    %         f(t-\tau), &t\geq \tau,\\
    %         0, &0\leq t < \tau;
    %     \end{cases}
    % \end{equation}
    %and belongs to the space $L^{\infty}_{loc}([0,+\infty);\mathcal{H}^p)$.

Next, we write $\widetilde\phi$ for the zero extension of $\phi$ over $\mathbb{R}$. (It is treated as a function of the time variable.) It is clear that $\widetilde{\phi}\in L^2(\mathbb{R};\mathcal{H}^{s-1})$. Furthermore, we have
\begin{align*}
    y(t-\tau; y_0, \phi) = \delta_{\tau}  * ( y(\cdot; y_0, \phi)|_{[0,+\infty)} ) (t)  +  \delta_{\tau} * \widetilde\phi(t),
    ~ t \geq 0.
\end{align*}
Then, it follows from the variation of constants formula that  $y(\cdot ; y_0, \phi)|_{[0,+\infty)}$  satisfies the following integral equation :
\begin{equation}\label{20230519-yb-EquationFor-y}
        y_{+}(t)    - a ( e^{ \cdot \Delta_D} * \delta_{\tau} ) * y_+(t)
        = e^{t \Delta_D} y_0
        +  a ( e^{ \cdot \Delta_D} *    \delta_{\tau} * \widetilde\phi )(t),\;\;t\geq 0,
        \end{equation}
with the unknown $y_{+}$ in $C([0,+\infty); L^2(\Omega))$.
% One can directly check that the above integral equation is equivalent to equation \eqref{20220510-yb-OriginalSystemWithTimeDelay} in the following sense: a function $y_+$ in $C([0,+\infty); \mathcal H^s)$ satisfies \eqref{20230519-yb-EquationFor-y} if and only if the following function
% \begin{align*}
%     y(t):=  \begin{cases}
%                 y_+(t), ~t\geq 0,
%                 \vspace{1ex}\\
%                 \phi(t), ~t\in (-\tau,0)
%             \end{cases}
% \end{align*}
% verifies equation \eqref{20220510-yb-OriginalSystemWithTimeDelay}.
% % Once this equation has a solution $y_+$ in $C([0,+\infty); \mathcal{H}^s)$, we find
% % \begin{equation}\label{zhang-8-24-1}
% %     \partial_t y_+ (t)  = \Delta y_+(t)
% %         + a [ (\delta_{\tau} * y_+)(t)  +   \widetilde\phi(t-\tau) ], ~ t> 0;
% %         ~~y_+(0)  =  y_0,
% % \end{equation}
% % which implies that $y_+$ also satisfied equation \eqref{20220510-yb-OriginalSystemWithTimeDelay} over $[0,+\infty)$, i.e.,
% % \begin{equation}\label{3.5-2-26-w}
% %     y_+(t) := y(t;y_0,\phi),  ~t\geq 0.
% % \end{equation}
% The rest proof is organized in the following several steps to study the integral equation \eqref{20230519-yb-EquationFor-y} and derive an expression for its solution.

\vskip 5pt
\noindent\textit{Step 2. Analysis of the integral equation \eqref{20230519-yb-EquationFor-y}}

First, we aim to put equation \eqref{20230519-yb-EquationFor-y} in a functional analysis framework. To this end, we arbitrarily fix  $T>0$ and then define a linear operator
$\mathcal G_T$ on $C([0,T]; L^2(\Omega))$ in the following manner:
given $f\in C([0,T]; L^2(\Omega))$, we set
\begin{equation}\label{yu-8-18-2}
   (\mathcal{G}_T f)(t):=    a (e^{ \cdot \Delta_D} * \delta_{\tau}) *  f
   (t),\;\;t\in[0,T].
  \end{equation}
Then, equation \eqref{20230519-yb-EquationFor-y} can be rewritten as:
\begin{align}\label{20241224-yb-IntegralEquation}
    (y_+|_{[0,T]})  -  \mathcal{G}_T (y_+|_{[0,T]})    =   F|_{[0,T]}
    ~\text{in}~
    C([0,T]; L^2(\Omega)),
\end{align}
where the function $F \in C([0,+\infty); \mathcal{H}^s)$ is defined by
\begin{align}\label{20241224-yb-NonhomogeneousTermInIntegralEquation}
    F(t) := e^{t \Delta_D} y_0   +    e^{ \cdot \Delta_D} *     ( a \delta_{\tau} * \widetilde\phi )(t),
    ~t \geq 0.
\end{align}
Notice that by
Lemma \ref{20241223-yb-lemma-SolutionsWithNonhomogeneousTermInFunctionalSetting},  the assumptions that $y_0\in\mathcal{H}^s$ and $\phi\in L^2((-\tau,0);\mathcal{H}^{s-1})$,  and the properties of the heat semigroup $\{ e^{t \Delta_D} \}_{t\geq0} $, one can verify that $F \in C([0,+\infty); \mathcal{H}^s)$.

% Furthermore, since $y_0\in\mathcal{H}^s$ and $\widetilde{\phi}\in L^2(\mathbb{R};\mathcal{H}^{s-1})$, we use the properties of the heat semigroup to obtain
% \begin{align}
%     F \in   X := C([0,+\infty); \mathcal H^s)
%     \cap  L^2_{loc}([0,+\infty); \mathcal H^{s+1}).
% \end{align}

Next, we claim that for each $m\geq 0$,
\begin{equation}\label{zhang-8-22-2}
        (I_m - \mathcal G_T)^{-1}  =  I_m +  \sum_{j=1}^{+\infty} \mathcal G_T^j
        ~\text{in}~
        \mathcal L( C([0,T]; \mathcal H^m) ),
\end{equation}
where $I_m$ is the identity operator over $C([0,T]; \mathcal H^m)$.
To this end, we fix an $m \geq 0$ and let
\begin{align*}
    K(t,\tau) := a (e^{ \cdot \Delta_D} * \delta_{\tau}) (t-\tau),~ t\geq \tau \geq 0.
\end{align*}
Then, by the definition of $\mathcal G_T$ (see \eqref{yu-8-18-2}), we obtain that for each $j\in\mathbb{N}^+$ and $f\in C([0,T]; \mathcal H^m)$,
\begin{align*}
   \| \mathcal G_T^j f \|_{C([0,T]; \mathcal H^m)}
   \leq& \sup_{0\leq t_1 \leq T}  \int_0^{t_1} dt_2 \cdots \int_0^{t_{j}}
        \| K(t_1,t_2) K(t_2,t_3) \cdots K(t_j,t_{j+1}) f(t_{j+1}) \|_{\mathcal H^m} dt_{j+1}
    \nonumber\\
   \leq&  \Big(
             \sup_{0\leq \tau \leq t \leq T} \| K(t,\tau) \|_{\mathcal L(\mathcal H^m)}^j
          \Big)
          \Big(
             \sup_{0\leq t_1 \leq T}
             \int_0^{t_1} dt_2  \cdots \int_0^{t_j} dt_{j+1}
          \Big)
        \|f\|_{C([0,T]; \mathcal H^m)}
    \nonumber\\
   \leq& a^j (T^j / j!)  \|f\|_{C([0,T]; \mathcal H^m)},
\end{align*}
which implies
\begin{equation*}%\label{yu-8-18-1}
    \|\mathcal{G}_T^j\|_{ \mathcal L( C([0,T];\mathcal H^m))}
    \leq (aT)^j / j!
    ~\text{for each}~
    j\in\mathbb{N}^+.
\end{equation*}
Therefore, the series
$
    \sum_{j=1}^{+\infty} \mathcal{G}_T^j
$
absolutely converges in $\mathcal{L}(C([0,T];\mathcal{H}^m))$. Furthermore, one can directly check that
$$
    (I_m - \mathcal{G}_T)     \bigg( I_m + \sum_{j=1}^{+\infty} \mathcal{G}_T^j  \bigg)
    =   I_m,
$$
which leads to \eqref{zhang-8-22-2}.

\vskip 5pt
\noindent\textit{Step 3. To solve the integral equation \eqref{20230519-yb-EquationFor-y} in $C([0,+\infty); \mathcal H^s)$}

Since $F \in C([0,+\infty); \mathcal{H}^s)$ (see \eqref{20241224-yb-NonhomogeneousTermInIntegralEquation}), we see from \eqref{zhang-8-22-2} that  equation \eqref{20241224-yb-IntegralEquation} (which is  a new form of equation \eqref{20230519-yb-EquationFor-y}) has a unique solution in $C([0,T]; L^2(\Omega))$, and is furthermore improved in  $C([0,T]; \mathcal{H}^s)$, i.e.,
\begin{align}\label{20241224-yb-FunctionalExpressionOfIntegralEquation}
    y_+|_{[0,T]} =  F|_{[0,T]} +  \sum_{j=1}^{+\infty} \mathcal{G}_T^j  (F|_{[0,T]})
    ~\text{in}~
    C([0,T]; \mathcal H^s).
\end{align}
At the same time, it follows from \eqref{yu-8-18-2} and \eqref{20241224-yb-NonhomogeneousTermInIntegralEquation} that for each $j\in\mathbb{N}^+$,
\begin{align*}
    \mathcal{G}_T^j  (F|_{[0,T]})    =&
    \underset{j}{ \underbrace{( a e^{ \cdot \Delta_D} * \delta_{\tau})
                *\cdots*
                ( a e^{\cdot \Delta_D} * \delta_{\tau} )} }
        * \left( e^{\cdot \Delta_D} y_0
        + a ( e^{ \cdot \Delta_D} *    \delta_{\tau} * \widetilde\phi )  \right)
        \nonumber\\
    =& a^j  \underset{j}{
                \underbrace{ \delta_{\tau}     *\cdots*    \delta_{\tau} }
            }
        * \underset{j+1}{
                \underbrace{ e^{\cdot \Delta_D}     *\cdots*    e^{\cdot \Delta_D} }
            } y_0
        +  a^{j}  \underset{j}{
                \underbrace{ \delta_{\tau}     *\cdots*    \delta_{\tau} }
            }
        * \underset{j+1}{
                \underbrace{ e^{\cdot \Delta_D}     *\cdots*    e^{\cdot \Delta_D} }
            } *  ( a \delta_{\tau} * \widetilde\phi )
    \nonumber\\
    =& a^j  \delta_{j\tau}
        * \underset{j+1}{
                \underbrace{ e^{\cdot \Delta_D}     *\cdots*    e^{\cdot \Delta_D} }
            } y_0
        +  a^{j}  \delta_{j\tau}
        * \underset{j+1}{
                \underbrace{ e^{\cdot \Delta_D}     *\cdots*    e^{\cdot \Delta_D} }
            } *  ( a \delta_{\tau} * \widetilde\phi )
    \nonumber\\
    =& a^j  \delta_{j\tau}     *  \Big( \frac{t^j}{j!} e^{t\Delta_D} \Big)  y_0
        +  a^{j}  \delta_{j\tau} *  \Big( \frac{t^j}{j!} e^{t\Delta_D} \Big)  *  ( a \delta_{\tau} * \widetilde\phi ),
\end{align*}
which implies that for each $j\in\mathbb{N}^+$,
\begin{align}\label{20241224-yb-ExpressionOfGeneralTermsInIntegralEquation}
    \mathcal{G}_T^j  (F|_{[0,T]})    =  \varPhi_j  y_0
        +  \varPhi_j  *  ( a \delta_{\tau} * \widetilde\phi )
        ~\text{over}~  [0,T],
\end{align}
where
\begin{equation}\label{yu-8-23-1}
    \varPhi_j(t)
    := \frac{1}{j!}(t-j\tau)^j  \chi_{[j\tau,+\infty)}(t)   e^{ (t-j\tau)^+ \Delta_D },
    ~t\geq 0.
\end{equation}

Now, by the definition of $\varPhi$ in \eqref{zhang-00-8-23-2}  and \eqref{yu-8-23-1}, we have
\begin{align*}
    \varPhi(t) = e^{t \Delta_D}  +  \sum_{j \geq 1}  \varPhi_j(t),~   t \geq 0.
\end{align*}
Since $T>0$ was arbitrarily taken, The above, together with  \eqref{20241224-yb-FunctionalExpressionOfIntegralEquation}, \eqref{20241224-yb-ExpressionOfGeneralTermsInIntegralEquation} and \eqref{20241224-yb-NonhomogeneousTermInIntegralEquation} , yields that the solution of equation \eqref{20230519-yb-EquationFor-y} satisfies
\begin{align}\label{20241224-yb-ExpressionOfIntegralEquationByFlowOperator}
    y_+ = \varPhi  y_0
        +  \varPhi  *  ( a \delta_{\tau} * \tilde\phi )
    ~\text{in}~
    C([0, +\infty); \mathcal H^s).
\end{align}

\vskip 5pt
\noindent\textit{Step 4. To verify  \eqref{zhang-00-8-23-3}}

Since  $y(\cdot, \cdot; y_0, \phi)|{[0,+\infty)}$ solves \eqref{20230519-yb-EquationFor-y} which has a unique solution,
we obtain  \eqref{zhang-00-8-23-3} from \eqref{20241224-yb-ExpressionOfIntegralEquationByFlowOperator}.
\vskip 5pt
In summary, the proof is completed.
\end{proof}

\section{Proofs of main results}\label{proof of main theorem}

\subsection{Proof of Theorem \ref{20230510-yb-theorem-PropagationOfSingularitiesOnlyWithInitialdata}}
%\label{section-ProofOfFirstMainTheorem}

This subsection is devoted to the proof of Theorem \ref{20230510-yb-theorem-PropagationOfSingularitiesOnlyWithInitialdata}. {\it We recall that
both $y(\cdot,\cdot; y_0,\phi)$ and $y(\cdot; y_0,\phi)$
are used to denote the solution of equation \eqref{20220510-yb-OriginalSystemWithTimeDelay}. They will be used in different places.}
We start with the following lemma.

\begin{lemma}\label{20240401-yubiao-SmoothSolutionWithSmoothHistory}
    Let $\phi \in C_0^{\infty}((-\tau,0) \times \Omega)$.      Then $y(\cdot,\cdot;0,\phi) \in  C^{\infty}( (-\tau,+\infty) \times \Omega)$.
\end{lemma}

\begin{proof}
We arbitrarily fix $x_0 \in \Omega$ and $t_0 > -\tau$. It suffices to show that the solution $y(\cdot,\cdot;0,\phi)$ is smooth at $(t_0, x_0)$.
There are only two cases for $t_0$: either $t_0\in (-\tau,0)$ or $t_0 \geq 0$. In the first case that $t_0\in (-\tau,0)$, we have
\begin{align*}
    y(t, x; 0,\phi) = \phi(t,x), ~t\in (-\tau, 0), ~x\in\Omega.
\end{align*}
Since $\phi \in C_0^{\infty}((-\tau,0) \times \Omega)$ and $t_0\in (-\tau,0)$, the above shows that the solution $y(\cdot,\cdot;0,\phi)$ is smooth at $(t_0, x_0)$.

In the second case $t_0 \geq 0$, we write $\tilde\phi$ and $\widetilde \varPhi$ for the zero extension of $\phi$ and $\varPhi$ in $\mathbb R \times \Omega$ and $\mathbb R$, respectively. (Here, $\varPhi$ is given by \eqref{zhang-00-8-23-2}. )
Then it follows from  \eqref{zhang-00-8-23-3} that
\begin{align*}
    y(t; 0,\phi) =&  a  \int_{-\tau}^{ \min\{t-\tau,0\}  }
        \varPhi(t-\tau-\gamma)   \phi(\gamma)  d\gamma
        \nonumber\\
=& a  \int_{\mathbb R} \widetilde\varPhi(t-\tau-\gamma)  \tilde\phi(\gamma) d\gamma
= a  \int_{\mathbb R} \widetilde\varPhi(\gamma)
\tilde\phi(t-\tau-\gamma) d\gamma,\;\;t\geq 0.
\end{align*}
However,  since $\phi \in C_0^{\infty}((-\tau,0) \times \Omega)$, we can find  $\varepsilon \in (0,\tau)$ such that
\begin{align*}
    y(t;0,\phi) = \phi(t)=0,\;\; ~t\in (-\varepsilon, 0),
\end{align*}
The above two equalities give
\begin{align}\label{4.1-1-1-w}
    y(t; 0,\phi) = a  \int_{\mathbb R} \widetilde\varPhi(\gamma)
\tilde\phi(t-\tau-\gamma) d\gamma,
~t > -\varepsilon.
\end{align}
At the same time, since $\phi \in C_0^{\infty}((-\tau,0) \times \Omega)$, we have
$\tilde \phi  \in C_0^{\infty} (\mathbb R; \mathcal H^s )$
    for each  $s \in \mathbb R$.
This, together with \eqref{4.1-1-1-w} and  \eqref{zhang-00-8-23-2}, yields
\begin{align*}
    y(\cdot; 0,\phi)  \in C^{\infty}( (-\varepsilon,+\infty); \mathcal H^s)
    ~\text{for each}~  s \in \mathbb R.
\end{align*}
Thus, the solution $y(\cdot, \cdot;0,\phi)$
is smooth in $(t_0, x_0)$ in the second case. This
completes the proof.
\end{proof}

We are now in a position to prove Theorem \ref{20230510-yb-theorem-PropagationOfSingularitiesOnlyWithInitialdata}.
\begin{proof}[Proof of Theorem \ref{20230510-yb-theorem-PropagationOfSingularitiesOnlyWithInitialdata}]

%Let $y_0 \in L^2(\Omega)$ and
%$\phi \in C_0^{\infty}((-\tau,0) \times \Omega)$.
% First of all, we claim
% \begin{align}
%     y(\cdot,\cdot;0,\phi) \in C^{\infty}( (-\tau,0) \times \Omega).
% \end{align}
% For this purpose, by the expression \eqref{zhang-00-8-23-3} for the solution $y$, we know
% \begin{align}
%     y(t,\cdot; 0,\phi) =  \int_{-\tau}^{ \min\{t-\tau,0\}  }
%         \varPhi(t-\tau-\gamma)   \phi(\gamma)  d\gamma.
% \end{align}
We arbitrarily fix    $x_0\in\Omega$ and  $j\in \mathbb N$,
and then we  arbitrarily fix $s,\alpha\in \mathbb N$ with $\alpha \leq j$.  First of all, we obtain from
 Lemma \ref{20240401-yubiao-SmoothSolutionWithSmoothHistory} that
\begin{align}\label{20240402-yubiao-SmoothSolutionWithOnlyHistory}
    y(\cdot;\cdot;0,\phi) \in  C^{\infty}( (-\tau,+\infty) \times \Omega).
\end{align}
The rest proof is organized into the following four steps.

\vskip 5pt
\noindent\textit{Step 1. To prove that $(i)\Rightarrow(ii)$.}

Suppose that the statement $(i)$ is true. To prove the statement $(ii)$, by contradiction, we suppose that it was not true, i.e.,
\begin{align*}
   y(\cdot,\cdot;y_0,\phi)  \in H^{\alpha, s+1+2(j-\alpha)}_{loc}(j\tau, x_0).
\end{align*}
Then by \eqref{yu-10-15-2}, there is  $\varepsilon \in (0,\tau)$ and  open neighborhood $\hat U \subset \Omega$  of $x_0$ such that
\begin{align*}
    y(\cdot,\cdot;y_0,\phi) \in H^{\alpha}
\big(
        (j\tau - \varepsilon, j\tau+\varepsilon); H^{s+1+2(j-\alpha)}_{loc}(\hat U) \big).
\end{align*}
From this and \eqref{20240402-yubiao-SmoothSolutionWithOnlyHistory}, we determine that
\begin{align}\label{20240402-yubiao-RegularityWithOnlyInitialData}
    y(\cdot,\cdot;y_0,0) = y(\cdot,\cdot;y_0,\phi)
    - y(\cdot,\cdot;0,\phi)
    \in H^{\alpha}
\big(
        (j\tau - \varepsilon, j\tau + \varepsilon); H^{s+1+2(j-\alpha)}_{loc}(\hat U) \big).
\end{align}

Next, we apply  Theorem \ref{20230510-yb-proposition-ExpressionForEvolutionWithTimeDelay}  (with $\phi=0$)  to obtain
\begin{align*}
    y(t, x; y_0, 0) %= \varPhi(t) y_0
    =&  \frac{a^j}{j!}  (t-j\tau)^j
        \chi_{[j\tau,+\infty)}(t)   [e^{(t-j\tau)^+ \Delta_D} y_0 ](x)
        \nonumber\\
     & +  \sum_{0 \leq k <j}
        \frac{a^k}{k!}  (t-k\tau)^k
         [e^{(t-k\tau) \Delta_D} y_0](x),\;\; \;t \in \big(
            (j-1) \tau, (j+1) \tau
        \big), ~x\in \Omega,
\end{align*}
with convention that  the sum is not there if its index set is empty.
Recalling \eqref{20240401-yb-WeightedSolutionsOfHeatEquation} for the definition of the function $f_k$, we see from the above that
\begin{align}\label{20240402-yubiao-TimeDelaySolutionAndHeatSolutions}
    y(t,x; y_0, 0) %= \varPhi(t) y_0
    =  \frac{a^j}{j!} f_j (t-j\tau,x)  +  \sum_{0 \leq k <j} \frac{a^k}{k!}
     f_k ( t - k\tau,x),\;\;t \in \big(
            (j-1) \tau, (j+1) \tau
        \big),
        ~x\in\Omega.
\end{align}

Now, according to  \eqref{20240402-yubiao-ContinuityAtInitialTimeForHeatSolution} (where $(\alpha,\beta) = (\hat \alpha,k)$), each $f_k(\cdot-k\tau,\cdot)$ in the sum of  \eqref{20240402-yubiao-TimeDelaySolutionAndHeatSolutions} is smooth over $\big(
            (j-1) \tau, (j+1) \tau
        \big)\times\Omega$.
%the following equality holds:
%\begin{align*}
  %  y(t,\cdot; y_0, 0) %= \varPhi(t) y_0
 %   = f_j (t-j\tau)  +  \sum_{0 \leq k <j} f_k ( t - k\tau),
 %   ~ (j-1)\tau  < t  < (j+1) \tau,
%\end{align*}
%where the functions $f_{k}$ ($k=0,\ldots,j$) are given by \eqref{20240401-yb-WeightedSolutionsOfHeatEquation}.  By  \eqref{20240402-yubiao-ContinuityAtInitialTimeForHeatSolution} (where $(\alpha,\beta) = (\hat \alpha,k)$), each function in the above sum is smooth.
This, together with \eqref{20240402-yubiao-TimeDelaySolutionAndHeatSolutions}  and \eqref{20240402-yubiao-RegularityWithOnlyInitialData}, implies
\begin{align*}
    f_j (\cdot-j\tau,\cdot)  \in H^{\alpha}
\big(
        (j\tau,j\tau+\varepsilon); H^{s+1+2(j-\alpha)}_{loc}(\hat U) \big).
\end{align*}
Then, by the definition of $f_j$ (see \eqref{20240401-yb-WeightedSolutionsOfHeatEquation}), we find that
\begin{align*}
  \mbox{the function} t\rightarrow   \partial_t^{\alpha} (t^j e^{t\Delta_D} y_0) \;(t\in(0,  \varepsilon))\;\;\mbox{belongs to}\;\; L^2 \big(
                (0, \varepsilon); H^{s+1+2(j-\alpha)}_{loc}(\hat U)
            \big).
\end{align*}
Then, we apply Proposition \ref{20231210-yb-proposition-ReverseRegularityForSolutionsOFHeatEquation}, where $\beta$, $T$, $U$ and $y$ are replaced by $j$, $\varepsilon$, $\hat U$ and $e^{t\Delta_D} y_0$,
respectively,
%\begin{align*}
 %   j, ~\varepsilon, ~\hat U, ~ e^{t\Delta_D} y_0,
%\end{align*}
%respectively,
to obtain
$y_0  \in H^s_{loc}(\hat U)$.
Now, by the definition of $H^s_{loc}(x_0)$ (see \eqref{yu-10-15-1}), we are led to a contradiction with the statement $(i)$. Therefore, the statement $(ii)$ is true.

\vskip 5pt
\noindent\textit{Step 2. To prove $(ii)\Rightarrow(iii)$.}

The statement $(iii)$ directly follows from the statement $(ii)$.

\vskip 5pt
\noindent\textit{Step 3. To prove $(iii)\Rightarrow(i)$.}

Assume that the statement $(iii)$ is true. To prove the statement $(i)$, by contradiction, we suppose that it was not true, i.e.,
\begin{align}\label{20240401-yubiao-ProofOfMainTheorem-StatementOne}
    y_0 \in H^s_{loc} (x_0).
\end{align}
In contradiction with the statement $(iii)$,
%by the definition of $\widehat{H}^{s+1+2j}_{loc}(j\tau,x_0)$ (see \eqref{20240331-yubiao-DefintionOfJointDerivatives}),
it suffices to show that for an arbitrarily fixed  $\hat \alpha \in \mathbb N \cap [0,j]$,
\begin{align}\label{20240402-yubiao-PrepareForStatementTwo}
    y(\cdot,\cdot; y_0,\phi) \in H^{\hat\alpha, s+1+2(j-\hat\alpha)}_{loc}(j\tau,x_0).
\end{align}
  For this purpose, it follows from \eqref{20240401-yubiao-ProofOfMainTheorem-StatementOne} that $y_0 \in H^{s}_{loc}(U)$ for some open subset $U$
  with  $x_0\in U\subset \Omega$ ,
%\begin{align*}
%    y_0 \in H^{s}_{loc}(U).
%\end{align*}
Then, for each natural number $k \leq j$, we apply \eqref{20240402-yubiao-ContinuityAtInitialTimeForHeatSolution} (where $(\alpha,\beta) = (\hat \alpha,k)$) to each function $f_k$ in \eqref{20240402-yubiao-TimeDelaySolutionAndHeatSolutions} to get that
\begin{align*}
    f_k  \in
    \begin{cases}
        C^{\infty}   ( ( 0, +\infty) \times \Omega )   ~\text{when}~   k < j ,
     \\     \\
        H^{\hat\alpha}(\mathbb R; H^{s+1+2(j-\hat\alpha)}_{loc} (U) )    ~\text{when}~   k = j .
    \end{cases}
\end{align*}
With \eqref{20240402-yubiao-TimeDelaySolutionAndHeatSolutions},
the above leads to
\begin{align*}
    y(\cdot,\cdot; y_0, 0) \in H^{\hat\alpha}_{loc}
    \big(
        ((j-1)\tau, (j+1) \tau); H^{s+1+2(j-\hat\alpha)}_{loc} (U)
    \big).
\end{align*}
This, along with \eqref{20240402-yubiao-SmoothSolutionWithOnlyHistory},
yields
\begin{align*}
    y(\cdot,\cdot;y_0,\phi) =& y(\cdot,\cdot;y_0,0)
        + y(\cdot,\cdot;0,\phi)
        \nonumber\\
        \in &  H^{\hat\alpha}_{loc}
    \big(
        ((j-1)\tau, (j+1) \tau); H^{s+1+2(j-\hat\alpha)}_{loc} (U)
    \big).
\end{align*}
With \eqref{yu-10-15-2}, the above gives \eqref{20240402-yubiao-PrepareForStatementTwo}, which contradicts the statement $(iii)$. Thus, the statement $(i)$ is true.

\vskip 5pt
\noindent\textit{Step 4. To show that $y(\cdot,\cdot;y_0,\phi)$ is smooth outside the set $\{j\tau\}_{j=0}^{+\infty} \times \Omega$.}

We arbitrarily fix $t_0 \in (-\tau,+\infty) \setminus \{j\tau\}_{j=0}^{+\infty}$ and $x_0 \in \Omega$.  We now show
that the solution $y(\cdot,\cdot;y_0,\phi)$ is smooth at the point $(t_0,x_0)$.
There are two cases for $t_0$: $t_0 \in (-\tau, 0)$ or $t_0 > 0$.
In the first case when $t_0 \in (-\tau, 0)$, we have
\begin{align*}
    y(t, x;y_0,\phi) = \phi(t,x), ~t\in (-\tau,0), x\in\Omega.
\end{align*}
Since $\phi \in C_0^{\infty}( (-\tau,0) \times \Omega)$ was assumed, it is clear that the solution $y(\cdot,\cdot;y_0,\phi)$ is smooth at the point $(t_0,x_0)$.

In the second case where $t_0 > 0$, we note that $t_0 \notin \{j\tau\}_{j=0}^{+\infty}$. Thus, we can let $j_0$ be the integer such that
$j_0 \tau  <  t_0  <  (j_0 + 1) \tau$.
By \eqref{zhang-00-8-23-3} (with $\phi=0$), we see  to obtain when
$t \in \big( j_0 \tau, (j_0+1) \tau \big)$,
\begin{align*}
    y(t; y_0, 0) %= \varPhi(t) y_0
    =  \frac{ a^{j_0} }{j_0!}  (t - j_0\tau)^{j_0}
         e^{(t-j_0\tau)^+ \Delta_D} y_0
 +  \sum_{0 \leq k < j_0}
        \frac{a^k}{k!}  (t-k\tau)^k
          e^{(t-k\tau) \Delta_D} y_0,\;\;t \in \big(
            j_0 \tau, (j_0+1) \tau
        \big).
          \end{align*}
With  the smoothing effect of the heat semigroup, the above result shows that the solution $y(\cdot,\cdot;y_0,0)$ is smooth at $(t_0,x_0)$. This, together with \eqref{20240402-yubiao-SmoothSolutionWithOnlyHistory}, implies that the function
\begin{align*}
    y(t,x;y_0,\phi)  = y(t, x;y_0,0)
    +  y(t, x;0,\phi),
    ~t \geq 0, x\in\Omega
\end{align*}
is smooth at $(t_0,x_0)$. This leads to the conclusion in Step 4 and completes the proof of Theorem \ref{20230510-yb-theorem-PropagationOfSingularitiesOnlyWithInitialdata}.
\end{proof}

\subsection{Proof of Theorem \ref{20230510-yb-theorem-PropagationOfSingularitiesForGeneralSolutions}}
%\label{section-ProofOfSecondMainTheorem}

This subsection is devoted to the proof of Theorem \ref{20230510-yb-theorem-PropagationOfSingularitiesForGeneralSolutions}, which requires the following result.

\begin{lemma}\label{20241219-yb-proposition-EquivalenceBetweenTwoInstants}
    Suppose that $y_0 \in L^2(\Omega)$ and $\phi \in L^2( (-\tau, 0) \times \Omega )$. Let $t_0\in[-\tau, 0)$, $x_0 \in \Omega$ and $j\in \mathbb N^+$. Then, given $r,s\in \mathbb N$, the following two statements are equivalent:
    \begin{itemize}
        \item[$(i)$] $y(\cdot,\cdot; y_0, \phi)  \in  \widehat{H}^{(r,s),2j}_{loc}(t_0 + j \tau,x_0)$;
        %$\displaystyle\bigcap_{\alpha,\beta\in\mathbb{N},\, 2\alpha+\beta=2j }
        %      H^{r+\alpha, s+\beta}_{loc}(t_0+j\tau,x_0)$;
        %
        \item[$(ii)$] $y(\cdot,\cdot; y_0, \phi)  \in  \widehat{H}^{(r,s),2}_{loc}(t_0 + \tau,x_0)$.
        %$\displaystyle\bigcap_{\alpha,\beta\in\mathbb{N},\, 2\alpha+\beta=2 }
        %      H^{r+\alpha, s+\beta}_{loc}(t_0+\tau,x_0)$.
    \end{itemize}
    Here,  $\widehat{H}^{(r,s),p}_{loc}(t,x)$ is given by \eqref{20241218-yb-JointLocalSobolevSpace}.
\end{lemma}

\begin{proof}
We arbitrarily fix $r,s\in \mathbb N$. We simply write
$y(\cdot,\cdot)$ for the solution $y(\cdot,\cdot; y_0, \phi)$.
It suffices to prove that for each $\hat j\in \mathbb N^+$,
\begin{align}\label{20241218-yb-EquivalenceBetweenTwoInstants}
    y   \in  \widehat{H}^{(r,s),2(\hat j+1)}_{loc}(t_0 + (\hat j+1) \tau,x_0)
    ~\Longleftrightarrow~
    y   \in  \widehat{H}^{(r,s),2\hat j}_{loc}(t_0 + \hat j \tau,x_0).
\end{align}
For this purpose, we arbitrarily fix  $j_0 \in \mathbb N^+$, then take  $\alpha,\beta \in \mathbb N$ such that
\begin{align}\label{20241218-yb-AlphaBeta}
    2\alpha + \beta = 2j_0.
\end{align}
Since $y(\cdot,\cdot)$ solves equation \eqref{20220510-yb-OriginalSystemWithTimeDelay}, we have
\begin{align*}
    \partial_t y(t,x)  - \Delta y(t,x) =  a y(t-\tau,x),
    ~t\in   \big((j_0-1)\tau, (j_0+1)\tau \big),
    ~x\in \Omega.
\end{align*}
Thus, we can apply Proposition \ref{2021218-yb-lemma-NonhomogenuousHeat}, where $t_0$, $(r,s)$ and $f$ are replaced by
\begin{align*}
    t_0 + (j_0+1) \tau,
    ~(r + \alpha, s + \beta)
    ~\text{and}~
     a y(\cdot - \tau),
\end{align*}
respectively, to obtain
\begin{align*}
  &  y(\cdot,\cdot)   \in  H^{(r+\alpha)+1, s+\beta}_{loc}  \big( t_0 + (j_0+1) \tau,x_0  \big)
    \cap
    H^{r+\alpha, (s+\beta)+2}_{loc} \big( t_0 + (j_0+1) \tau,x_0  \big)
    \nonumber\\
   & \quad\quad\quad\quad\quad
    \Longleftrightarrow  ~
    y(\cdot -\tau,\cdot)   \in  H^{r+\alpha, s+\beta}_{loc} \big( t_0 + (j_0+1) \tau,x_0  \big).
\end{align*}
Because $\alpha$ and $\beta$ can be any non negative numbers  with \eqref{20241218-yb-AlphaBeta}, the above, together with \eqref{20241218-yb-JointLocalSobolevSpace}, implies
\begin{align*}
    y(\cdot,\cdot)   \in  \widehat{H}^{(r,s),2(j_0+1)}_{loc} \big( t_0 + (j_0+1) \tau,x_0  \big)
    ~\Longleftrightarrow~
    y(\cdot -\tau,\cdot)   \in  \widehat{H}^{(r,s),2j_0}_{loc} \big( t_0 + (j_0+1) \tau,x_0  \big),
\end{align*}
which gives \eqref{20241218-yb-EquivalenceBetweenTwoInstants} with $\hat j=j_0$. This completes the proof.
\end{proof}

We are now in a position to prove Theorem \ref{20230510-yb-theorem-PropagationOfSingularitiesForGeneralSolutions}.

\begin{proof}[Proof of Theorems \ref{20230510-yb-theorem-PropagationOfSingularitiesForGeneralSolutions}]
We arbitrarily fix $x_0\in \Omega$, $j \in \mathbb N^+$, and $r,s \in \mathbb N$.
First, we claim
\begin{align}\label{20241219-yb-EquivalenceBetweenFirstTwoInstants}
    \phi \in H^{r,s}_{loc}(t_0, x_0)
    \Longleftrightarrow
    y(\cdot,\cdot; y_0, \phi) \in
    \widehat{H}^{(r,s),2}_{loc}(t_0 + \tau, x_0).
    % y(\cdot)   \in  H^{r+1, s}  \big( t_0 + \tau,x_0  \big)
    % \cap
    % H^{r, s+2}_{loc} \big( t_0 + \tau,x_0  \big).
\end{align}
Indeed, by \eqref{20220510-yb-OriginalSystemWithTimeDelay}, we see that $y(\cdot,\cdot; y_0, \phi)$ satisfies
\begin{align*}
    \partial_t y(t,x)  - \Delta y(t,x) = a \phi(t-\tau,x),\;\;\;
    t\in (0, \tau), x\in \Omega.
\end{align*}
Then, we can apply Proposition \ref{2021218-yb-lemma-NonhomogenuousHeat},  where $t_0$ and $f$ are replaced by $t_0 + \tau$ and $a\phi(\cdot-\tau,\cdot)$,
%\begin{align*}
 %     t_0 +  \tau
  %  ~\text{and}~
   %  a y(\cdot - \tau),
%\end{align*}
respectively, to obtain
\begin{align*}
    y(\cdot,\cdot; y_0, \phi)   \in  H^{r+1, s}  ( t_0 + \tau, x_0)
    \cap
    H^{r, s+2}_{loc}  (t_0 + \tau, x_0)
    \Longleftrightarrow  ~
    \phi  \in  H^{r, s}_{loc} (t_0, x_0).
\end{align*}
With the definition of $\widehat{H}^{(r,s),2}_{loc}(t_0+\tau, x_0)$ in \eqref{20241218-yb-JointLocalSobolevSpace}, the above leads to \eqref{20241219-yb-EquivalenceBetweenFirstTwoInstants}.

Next, by Lemma \ref{20241219-yb-proposition-EquivalenceBetweenTwoInstants}, we find
\begin{align*}
    y(\cdot,\cdot; y_0, \phi) \in
    \widehat{H}^{(r,s),2}_{loc}(t_0 + \tau, x_0)
    \Longleftrightarrow
    y(\cdot,\cdot; y_0, \phi) \in
    \widehat{H}^{(r,s),2j}_{loc}(t_0 + j\tau, x_0).
\end{align*}
This, together with \eqref{20241219-yb-EquivalenceBetweenFirstTwoInstants}, gives the equivalence between the statements $(i)$ and $(ii)$ in this theorem. Thus, we finish the proof of Theorems \ref{20230510-yb-theorem-PropagationOfSingularitiesForGeneralSolutions}.
 \end{proof}

\subsection{
Proof of Theorem \ref{20240406-yubiao-Theorem-SingularitiesOfHistoricalValueAtEndpoints}
}
%\label{section-FurtherStudies}

This subsection presents the proof of Theorem \ref{20240406-yubiao-Theorem-SingularitiesOfHistoricalValueAtEndpoints}.
We first give the following lemma,
which is an analogue of Proposition \ref{2021218-yb-lemma-NonhomogenuousHeat} with an emphasis on regularities up to the boundaries of time intervals.

\begin{lemma}\label{2021220-yb-lemma-IntialTimeCaseForNonhomogenuousHeat}
    Let $T>0$ and let $U$ be an open non-empty subset of $\mathbb R^n$. Suppose that  two distributions
    $y \in C( [0, T) ;   \mathcal{D}'(U) )$
    and
    $f\in \mathcal{D}'((0,T) \times U)$ satisfy
\begin{align}\label{20241218-yubiao-IntialTimeCaseNonhomogenuousHeatEquation}
\partial_t y  - \Delta y = f ~\text{in}~ (0, T) \times U.
% \left\{
%     \begin{array}{ll}
%          \partial_t y  - \Delta y = f &\text{in}~ (T_1,T_3) \times \Omega,\\
%           y = 0                       &\text{on}~ (T_1,T_3) \times \partial\Omega.
%     \end{array}
% \right.
\end{align}
Then, given $r,s\in \mathbb N$ and $x_0 \in U$,
the following statements are equivalent:
\begin{itemize}[leftmargin=4em]
    \item[$(i)$]  $y \in H^{r+1,s}_{loc,+}(0,x_0) \cap H^{r,s+2}_{loc,+}(0,x_0)$.
    \item[$(ii)$] It holds that
\begin{align}\label{20241220-yb-CompatibilityOfRegularitiesAtInitialTime}
   f\in H^{r,s}_{loc,+}(0,x_0)
   ~\text{and}~
   \mathcal{G}_k   \in  H^{s+1}_{loc} (x_0)
   ~\text{for each integer}~   k \in [0,r],
   % ~\text{and}~
   % \mathcal{G}_k   \in  H^{s+3}_{loc} (x_0)
   % ~\text{when}~   k  <  r.
\end{align}
where   $\{ \mathcal{G}_k \}_{k=0}^r$ is  a sequence of distributions, given recursively  in the following manner:
\begin{align}\label{20241220-yb-DefinitionsOfTermsInCompatibilityOfRegularitiesAtInitialTime}
    \mathcal{G}_k  :=
    \begin{cases}
        y|_{t=0}   ~~\text{when}~ k=0,  \\
        \partial_t^{k-1} f(0, \cdot)  + \Delta  \mathcal{G}_{k-1}
        ~~\text{when}~   k \in (0,r],
    \end{cases}
   ~\text{defined locally around}~  x_0.
\end{align}
\end{itemize}
\end{lemma}

\begin{proof}
We arbitrarily fix $r,s\in \mathbb N$ and $x_0 \in U$.
The proof is divided into the following two steps.

\vskip 5pt
\noindent\textit{Step 1. We prove that $(i) \Leftrightarrow (ii)$
for the case $r=0$.}

We first show that $(i)\Rightarrow(ii)$ for this case.
Suppose that $(i)$ with $r=0$ holds. Then
%Assume that $r=0$. We start with proving the necessity. Suppose that
%\begin{align*}
 %   y \in H^{1,s}_{loc,+}(0,x_0) \cap H^{0,s+2}_{loc,+}(0,x_0).
%\end{align*}
by \eqref{yu-10-15-3}, there is  $\varepsilon_1 \in (0,T)$ and an open set $U_1$, with $x_0\in U_1\subset U$, such that
\begin{align}\label{20241222-yb-RegularityOfSolutionAroundInitialTime-InCase-r0}
    y \in   H^{1}((0,\varepsilon_1); H^{s}(U_1) )
    \cap L^2((0,\varepsilon_1); H^{s+2}(U_1)),
    %H^{1,s}((0,\varepsilon_1) \times U_1)      \cap H^{0,s+2}((0,\varepsilon_1) \times U_1) .
\end{align}
which yields $\partial_t y,  \Delta y \in  L^2((0,\varepsilon_1); H^{s}(U_1))$.
  Thus, we have
\begin{align}\label{20241222-yb-RegularityOfNonhomogeneousTermAroundInitialTime-InCase-r0}
    f = \partial_t y - \Delta y \in L^2((0,\varepsilon_1); H^s(U_1) )
    = H^{0,s}((0,\varepsilon_1) \times U_1),
\end{align}
which gives the first statement in \eqref{20241220-yb-CompatibilityOfRegularitiesAtInitialTime} with $r=0$.

Now, we claim that there is a function $y_1$ such that
\begin{align}\label{20241222-yb-RegularityOfConstructedHeatEquationWithSameNonhomogeneousTerm}
    y_1  \in H^{1}((0,\varepsilon_1); H^{s}_{loc}(U_1) )
    \cap L^2((0,\varepsilon_1); H^{s+2}_{loc}(U_1)),
\end{align}
and
\begin{align}\label{20241222-yb-ConstructedHeatEquationWithSameNonhomogeneousTerm}
    \partial_t y_1   -  \Delta y_1  =  f
     ~\text{in}~   (0,\varepsilon_1) \times U_1;
     ~~ y_1|_{t=0} = 0  ~\text{in}~    U_1.
\end{align}
To this end, we write $\tilde f_{\varepsilon_1, U_1}$ for the zero extension of $f|_{(0,\varepsilon_1) \times U_1}$ over $\mathbb R \times \mathbb R^n$. Then it follows from \eqref{20241222-yb-RegularityOfNonhomogeneousTermAroundInitialTime-InCase-r0} that
\begin{align}\label{20241222-yb-ExtentdedNonhomogeneousTerm}
   \tilde f_{\varepsilon_1, U_1}  \in   L^2(\mathbb R; H^s(U_1) \cap L^2(\mathbb R^n) ).
\end{align}
Let $z_1 \in C([0, +\infty); L^2(\mathbb R^n))$ be the solution to the following heat equation:
 \begin{align*}
    \partial_t z_1   -  \Delta z_1  =  \tilde{f}_{\varepsilon_1, U_1}
     ~\text{in}~   (0,+\infty) \times \mathbb R^n;
     ~~ z_1|_{t=0} = 0  ~\text{in}~    \mathbb R^n.
 \end{align*}
 Let $\tilde z_1$ be the zero extension of $z_1$ over $\mathbb R \times \mathbb R^n$. Then, from the above, one can easily check that
  \begin{align*}
     \partial_t \tilde z_1   -  \Delta \tilde z_1  =  \tilde{f}_{\varepsilon_1, U_1}
     ~\text{in}~   \mathbb R \times  \mathbb R^n\;\;\mbox{in the sense of distribution}.
 \end{align*}
Since $\tilde f_{\varepsilon_1, U_1} \in L^2(\mathbb R; H^s(U_1) )$ (see \eqref{20241222-yb-ExtentdedNonhomogeneousTerm}), we can
apply Proposition \ref{2021218-yb-lemma-NonhomogenuousHeat} (where $r=0$ and $(t_0,x_0)$ is any point in $\mathbb R \times U_1$) to obtain $\tilde z_1  \in H^{1,s}_{loc}(\mathbb R \times U_1)\cap H^{0,s+2}_{loc}(\mathbb R \times U_1)$.
 Then, the function $y_1 = \tilde z_1 |_{(0,\varepsilon_1) \times U_1}$
 satisfies  \eqref{20241222-yb-RegularityOfConstructedHeatEquationWithSameNonhomogeneousTerm} and  \eqref{20241222-yb-ConstructedHeatEquationWithSameNonhomogeneousTerm}. So, the above claim is true.

Now, we see from \eqref{20241218-yubiao-IntialTimeCaseNonhomogenuousHeatEquation} and  \eqref{20241222-yb-ConstructedHeatEquationWithSameNonhomogeneousTerm} that
  \begin{align}\label{6.12-1-4-w}
     \partial_t (y - y_1)   -  \Delta (y-y_1)  =  0
     ~\text{in}~   (0,\varepsilon_1) \times U_1;
     ~~ (y-y_1)|_{t=0} = y|_{t=0}   ~\text{in}~    U_1.
 \end{align}
 At the same time, it follows from \eqref{20241222-yb-RegularityOfSolutionAroundInitialTime-InCase-r0}
 and \eqref{20241222-yb-RegularityOfConstructedHeatEquationWithSameNonhomogeneousTerm} that
  \begin{align}\label{6.13-1-4-w}
       y - y_1 \in   L^2((0,\varepsilon_1); H^{s+2}_{loc}(U_1) ).
 \end{align}
By \eqref{6.12-1-4-w} and \eqref{6.13-1-4-w}, we can use  Proposition \ref{20231210-yb-proposition-ReverseRegularityForSolutionsOFHeatEquation} (where $(\alpha,\beta, T, U, y)$ is replaced by $(0,0, \varepsilon_1, U_1, y - y_1)$) to obtain
 \begin{align*}
     y|_{t=0}  =   (y-y_1)|_{t=0}
      \in H^{s+1}_{loc} (U_1).
 \end{align*}
Since $x_0 \in U_1$, the above, together with the definition of $\mathcal{G}_0$ in \eqref{20241220-yb-DefinitionsOfTermsInCompatibilityOfRegularitiesAtInitialTime}, implies
\begin{align*}
    \mathcal{G}_0 = y_0   ~\text{in}~  H^{s+1}_{loc}(x_0),
\end{align*}
 which, along with \eqref{20241222-yb-RegularityOfNonhomogeneousTermAroundInitialTime-InCase-r0},
 gives   \eqref{20241220-yb-CompatibilityOfRegularitiesAtInitialTime} for the case $r=0$.
 Hence, we have proven $(i)\Rightarrow(ii)$ for the case $r=0$.

 Next, we will show that $(ii)\Rightarrow(i)$ for the case $r=0$.
 For this purpose, we suppose that $(ii)$ holds for $r=0$.
 Then by  \eqref{yu-10-15-3}, there is  $\varepsilon_2 \in (0,T)$ and an open set $U_2$ with  $x_0\in U_2\subset U$ such that
\begin{align}\label{20241222-yb-RegularityOfSolutionAroundInitialTime-InSufficiencyInCase-r0}
    f \in L^2((0,\varepsilon_2); H^s(U_2) )
    ~\text{and}~
   y|_{t=0}   \in  H^{s+1} (U_2).
\end{align}
 By the similar arguments as  those used in the proof of \eqref{20241222-yb-RegularityOfConstructedHeatEquationWithSameNonhomogeneousTerm}, one can
 obtain  a function $y_2$ such that
\begin{align}\label{20241222-yb-Repreated-RegularityOfConstructedHeatEquationWithSameNonhomogeneousTerm}
    y_2  \in H^{1}((0,\varepsilon_2); H^{s}_{loc}(U_2) )
    \cap L^2((0,\varepsilon_2); H^{s+2}_{loc}(U_2)),
\end{align}
and
\begin{align*}%\label{20241222-yb-Repreated-ConstructedHeatEquationWithSameNonhomogeneousTerm}
    \partial_t y_2   -  \Delta y_2  =  f
     ~\text{in}~   (0,\varepsilon_2) \times U_2;
     ~~ y_2|_{t=0} = 0  ~\text{in}~    U_2.
\end{align*}
Then, it follows from \eqref{20241218-yubiao-IntialTimeCaseNonhomogenuousHeatEquation} and \eqref{20241222-yb-RegularityOfSolutionAroundInitialTime-InSufficiencyInCase-r0} that
  \begin{align*}
     \partial_t (y - y_2)   -  \Delta (y-y_2)  =  0
     ~\text{in}~   (0,\varepsilon_2) \times U_2;
     ~~ (y-y_2)|_{t=0} = y|_{t=0}   ~\text{in}~  H^{s+1} (U_2).
 \end{align*}
 Thus, we apply Proposition \ref{20231210-yb-proposition-ReverseRegularityForSolutionsOFHeatEquation} (where $(\alpha,\beta, T, U, y)$ is replaced by $(0,0, \varepsilon_2, U_2, y - y_2)$) to obtain
 \begin{align*}
      y-y_2        \in   L^2((0,\varepsilon_2); H^{s+2}_{loc}(U_2)).
 \end{align*}
Since $y - y_2$ satisfies the heat equation over $(0,\varepsilon_2) \times U_2$, the above implies
 \begin{align*}
      y-y_2        \in  H^{1}((0,\varepsilon_2); H^{s}_{loc}(U_2) )
    \cap L^2((0,\varepsilon_2); H^{s+2}_{loc}(U_2)).
 \end{align*}
 Since $x_0 \in U_2$, the above, together with \eqref{20241222-yb-Repreated-RegularityOfConstructedHeatEquationWithSameNonhomogeneousTerm} and \eqref{yu-10-15-3}, yields
\begin{align*}
    y = (y-y_2) + y_2 \in H^{1,s}_{loc,+}(0,x_0) \cap H^{0,s+2}_{loc,+}(0,x_0).
\end{align*}
Hence, we have proven $(ii)\Rightarrow(i)$, and further
$(i) \Leftrightarrow (ii)$
for the case $r=0$.

\vskip 5pt
\noindent\textit{Step 2. We prove $(i) \Leftrightarrow (ii)$ for
the case $r\geq 1$.}

First, we will show  $(i)\Rightarrow(ii)$ for the case $r\geq 1$. To this end, we suppose that $(i)$ holds for $r\geq 1$.
Then by \eqref{yu-10-15-3}, there is  $\varepsilon_3 \in (0,T)$ and an open set $U_3$ with $x_0\in U_3\subset U$ such that
\begin{align*}%\label{20241220-yb-RegularityOfSolutionAroundInitialTime}
    y \in H^{r+1}((0,\varepsilon_3); H^{s}(U_3) )
    \cap H^{r}((0,\varepsilon_3); H^{s+2}(U_3) ) .
\end{align*}
We arbitrarily fix $r\geq 1$ and an
 integer $k\in [0,r]$. Then the above yields
\begin{align}\label{6.16-1-4-w}
    \partial_t^k y \in  H^{1}((0,\varepsilon_3); H^{s}(U_3) )  \cap  L^2((0,\varepsilon_3); H^{s+2}(U_3) ).
\end{align}
By \eqref{6.16-1-4-w}, we have
\begin{align}\label{6.17-1-4-w}
    \partial_t^k y \in  C([0, \varepsilon_3); H^{s}(U_3)).
\end{align}
At the same time, it follows from \eqref{20241218-yubiao-IntialTimeCaseNonhomogenuousHeatEquation} that
\begin{align}\label{6.18-1-4-w}
\partial_t (\partial_t^k y)  - \Delta (\partial_t^k y) = \partial_t^k f
~\text{in}~ (0, \varepsilon_3) \times U_3.
\end{align}
With \eqref{6.16-1-4-w}, \eqref{6.17-1-4-w} and \eqref{6.18-1-4-w},
we can use the conclusion in Step 1, where $r=0$ and
 $(T, U, y, f)$ is replaced by $( \varepsilon_3, U_3, \partial_t^k y, \partial_t^k f)$,
to obtain
\begin{align}\label{20241222-yb-HighOrderRegularityInNecessityInGeneral-r}
    \partial_t^k f \in H^{0, s}_{loc,+}(0,x_0)
    ~\text{and}~
    \partial_t^k y |_{t=0} \in H^{s+1}_{loc}(x_0).
\end{align}
Meanwhile, direct computations show
$\partial_t^k y |_{t=0}  =  \mathcal{G}_k$  near  $x_0$,
which, together with \eqref{20241222-yb-HighOrderRegularityInNecessityInGeneral-r}, leads to \eqref{20241220-yb-CompatibilityOfRegularitiesAtInitialTime} with the above $r$.
This gives $(i)\Rightarrow(ii)$ for the case $r\geq 1$.

Next, we will
prove $(ii)\Rightarrow(i)$ for the case $r\geq 1$.
To this end, we suppose that $(ii)$ with $r\geq 1$ is true. Then by \eqref{yu-10-15-3}, there is $\varepsilon_4 \in (0,T)$ and an open set $U_4$ with $x_0\in U_4\subset U$ such that
\begin{align}\label{20241222-yb-RegularityOfSolutionAroundInitialTime-InSufficiencyInGeneral-r}
    f \in H^{r}( (0,\varepsilon_4); H^{s}(U_4) )
    ~\text{and}~
   \mathcal{G}_k  \in  H^{s+1} (U_4)
     ~\text{when}~ k \in [0,r].
\end{align}
Since  $r\geq 1$,  the above  implies that $f \in C^{r-1}( [0,\varepsilon_4); H^{s}(U_4) )$.
We arbitrarily fix  an integer $k\in [0,r]$.
Then, by making use of  \eqref{20241218-yubiao-IntialTimeCaseNonhomogenuousHeatEquation} iteratively, we see that
\begin{align*}
    \partial_t^k y  \in  C( [0,\varepsilon_4); \mathcal{D}'(U_4) ),
\end{align*}
and further obtain that
\begin{align*}
\partial_t (\partial_t^k y)  - \Delta (\partial_t^k y) = \partial_t^k f
~\text{in}~ (0, \varepsilon_4) \times U_4;
~~
(\partial_t^k y) |_{t=0}  =  \mathcal{G}_k   ~\text{in}~   U_4.
\end{align*}
By this and \eqref{20241222-yb-RegularityOfSolutionAroundInitialTime-InSufficiencyInGeneral-r},
we can use the conclusion in  Step 1,  where $r=0$ and $(T, U, y, f)$ is replaced by $ (\varepsilon_4, U_4, \partial_t^k y, \partial_t^k f)$,
to obtain
\begin{align*}
    \partial_t^k y \in  H^{1,s}_{loc,+}(0,x_0) \cap H^{0,s+2}_{loc,+}(0,x_0).
    %H^{1}((0,\varepsilon_4); H^{s}_{loc}(U_4) )  \cap  L^2((0,\varepsilon_4); H^{s+2}_{loc}(U_4) ).
\end{align*}
This implies
\begin{align*}
    y \in  H^{r+1,s}_{loc,+}(0,x_0) \cap H^{r,s+2}_{loc,+}(0,x_0).
    %H^{r+1}((0,\varepsilon_4); H^{s}_{loc}(U_4) )     \cap H^{r}((0,\varepsilon_4); H^{s+2}_{loc}(U_4) ) .
\end{align*}
Hence, we have proven $(ii)\Rightarrow(i)$, and further obtain $(i) \Leftrightarrow (ii)$ for the case $r\geq 1$.

\vskip 5pt
In summary, we complete the proof.
\end{proof}

Next, we give the proof of
Theorem \ref{20240406-yubiao-Theorem-SingularitiesOfHistoricalValueAtEndpoints}.

\begin{proof}[Proof of Theorem \ref{20240406-yubiao-Theorem-SingularitiesOfHistoricalValueAtEndpoints}]

We arbitrarily fix $j\in  \mathbb N^+$. We simply write $y(\cdot,\cdot)$
for the solution $y(\cdot,\cdot; y_0, \phi)$.
The proof will be organized into several steps.

\vskip 5pt
\noindent\textit{Step 1. We  prove
\begin{align*}%\label{20241220-yb-PreBackwardInSectionOnExtension}
    y   \in  \widehat{H}^{(r,s),2j}_{loc}(-\tau + j \tau,x_0)
    ~\Longleftrightarrow~
    y   \in  H^{r+1, s}_{loc}(0,x_0)   \cap H^{r, s+2}_{loc}(0,x_0).
\end{align*}
}

This is a consequence of Lemma \ref{20241219-yb-proposition-EquivalenceBetweenTwoInstants} (where  $t_0=-\tau$) and the definition of $\widehat{H}^{(r,s),2}_{loc}(0,x_0)$ in \eqref{20241218-yb-JointLocalSobolevSpace}.

\vskip 5pt
\noindent\textit{Step 2. We prove  that $y\in H^{r+1, s}_{loc}(0,x_0)   \cap H^{r, s+2}_{loc}(0,x_0)$ if and only if for each $k \in \mathbb N \cap [0,r]$,
\begin{align}\label{20241222-yb-DecomposedInformationOfSolutionAtInitialTime}
    y \in H^{r+1, s}_{loc,\pm}(0,x_0)   \cap  H^{r, s+2}_{loc,\pm}(0,x_0);\;\;
    \lim_{t\rightarrow 0^-} \partial_t^k y(t,\cdot) = \lim_{t\rightarrow 0^+} \partial_t^k y(t,\cdot)
    ~\text{near}~ x_0.
\end{align}
Here,
$H^{r+1, s}_{loc,\pm}(0,x_0) := H^{r+1, s}_{loc,+}(0,x_0)\cap  H^{r+1, s}_{loc, -}(0,x_0)$;
$H^{r, s+2}_{loc,\pm}(0,x_0) := H^{r, s+2}_{loc,+}(0,x_0)\cap  H^{r, s+2}_{loc, -}(0,x_0)$.
}

With the aid of \eqref{yu-10-15-3} and \eqref{yu-10-15-2}, we can directly
check \eqref{20241222-yb-DecomposedInformationOfSolutionAtInitialTime}.

\vskip 5pt
\noindent\textit{Step 3. We show that $y\in H^{r+1,s}_{loc,+}(0,x_0) \cap H^{r,s+2}_{loc,+}(0,x_0)$ if and only if $\phi \in H^{r,s}_{loc,+}(-\tau,x_0)$ and   condition $(2)$ of $(ii)$ in  Theorem \ref{20240406-yubiao-Theorem-SingularitiesOfHistoricalValueAtEndpoints}
holds.
}

By equation \eqref{20220510-yb-OriginalSystemWithTimeDelay} , we have
\begin{align*}
    % y|_{(-\tau,0) \times \Omega}   =  \phi
    % ~\text{and}~
    \partial_t y  - \Delta y =  a \phi(\cdot - \tau, \cdot)
    ~\text{in}~
    (0, \tau) \times \Omega.
\end{align*}
Then, we can apply Lemma \ref{2021220-yb-lemma-IntialTimeCaseForNonhomogenuousHeat} (where $(U,T,f)$ is replaced by $(\Omega, \tau, a \phi(\cdot - \tau, \cdot) )$) to get
the conclusion in Step 3.

% First, by equation \eqref{20220510-yb-OriginalSystemWithTimeDelay} , we have
% \begin{align*}
%     % y|_{(-\tau,0) \times \Omega}   =  \phi
%     % ~\text{and}~
%     \partial_t y  - \Delta y =  a \phi(\cdot - \tau, \cdot)
%     ~\text{in}~
%     (0, \tau) \times \Omega.
% \end{align*}
% Second, {\color{green}if  \eqref{20241222-yb-RegularityOfAllDerivativesOfSolution-gk} holds, then, by \eqref{20241222-yb-DefinitionsOf-gk} and some computations, we have}
% \begin{align*}
%     g_k = \mathcal{G}_k   ~\text{near}~  x_0,\;\;\mbox{for each}\;\;    k \in \mathbb N \cap [0,r]
% \end{align*}
% with  $f(\cdot,\cdot) = a \phi(\cdot - \tau,\cdot) $. Now, by the above two conclusions, we can apply Lemma \ref{2021220-yb-lemma-IntialTimeCaseForNonhomogenuousHeat} (where $(T,f)$ is replaced by $(\tau, a \phi(\cdot - \tau, \cdot) )$) to get
% the conclusion in Step 3.

\vskip 5pt
\noindent\textit{Step 4. We prove  that if condition $(1)$ in
$(ii)$ is true, then
\begin{align*}
   \lim_{t\rightarrow 0^-}     \partial^k_t y(t,\cdot)
   = \partial^k_t\phi(0,\cdot)
        ~\text{and}~
   \lim_{t\rightarrow 0^+} \partial^k_t y(t,\cdot)
   = g_k  ~\mbox{near}~ x_0,
       ~\text{for each}~   k \in \mathbb N \cap [0,r].
\end{align*}
}

By equation \eqref{20220510-yb-OriginalSystemWithTimeDelay}, we have
\begin{align*}
    y|_{(-\tau,0) \times \Omega}   =  \phi
    ~\text{and}~
    \partial_t y  - \Delta y =  a \phi(\cdot - \tau)
    ~\text{in}~
    (0, \tau) \times \Omega.
\end{align*}
By these and the definitions of $\{g_k\}_{k=0}^r$ in \eqref{20241222-yb-DefinitionsOf-gk},
after some computations, we can obtain the conclusions in Step 4.

\vskip 5pt
\noindent\textit{Step 5. We prove $(i)\Rightarrow(ii)$.
}

We assume that $(i)$ holds. To prove the statement $(ii)$, by contradiction, we suppose that it was not true, i.e., all the conditions $(1)$, $(2)$ and $(3)$, presented in the statement $(ii)$, are true.

Several facts are stated as follows: First, since $y|_{(-\tau,0) \times \Omega}   =  \phi$, we deduce from condition
$(1)$ that
\begin{align}\label{6.22-1-5-wgs}
    y \in H^{r+1, s}_{loc,-}(0,x_0)   \cap  H^{r, s+2}_{loc,-}(0,x_0).
\end{align}
Second, by conditions $(1)$ and $(2)$, we get  from the conclusion in Step 3 that
\begin{align}\label{6.23-1-5-wgs}
    y \in H^{r+1, s}_{loc,+}(0,x_0)   \cap  H^{r, s+2}_{loc,+}(0,x_0).
\end{align}
Third, by the conclusion in Step 4 and condition $(3)$, we have
\begin{align}\label{6.24-1-5-wgs}
 \lim_{t\rightarrow 0^-} \partial_t^k y(t,\cdot) = \lim_{t\rightarrow 0^+} \partial_t^k y(t,\cdot)
    ~\text{near}~ x_0,
    ~\text{for each}~    k \in \mathbb N \cap [0,r].
   \end{align}

Now, by \eqref{6.22-1-5-wgs}, \eqref{6.23-1-5-wgs}  and \eqref{6.24-1-5-wgs}, we can use the conclusion in  Step 2 to see
\begin{align*}
    y \in H^{r+1, s}_{loc}(0,x_0)   \cap  H^{r, s+2}_{loc}(0,x_0).
\end{align*}
By the above, we can use the conclusion in Step 1 to find
\begin{align*}
    y   \in  \widehat{H}^{(r,s),2j}_{loc}(-\tau + j \tau,x_0),
\end{align*}
which contradicts the statement $(i)$. Therefore, the statement $(ii)$ is true.

\vskip 5pt
\noindent\textit{Step 6. We prove $(ii)\Rightarrow(i)$.
}

Assume that the statement $(ii)$ is true. That is, at least one of conditions $(1)$, $(2)$, and $(3)$ is not true.
To prove  the statement $(i)$, by contradiction, we suppose that it was not true, i.e.,
    $y \in \widehat{H}^{(r,s),2j}_{loc}( -\tau + j\tau, x_0)$.
This, along with the conclusion in  Step 1, yields
\begin{align*}
    y \in H^{r+1, s}_{loc}(0,x_0)   \cap H^{r, s+2}_{loc}(0,x_0).
\end{align*}
With the conclusion of Step 2, the above leads to \eqref{20241222-yb-DecomposedInformationOfSolutionAtInitialTime}.
Then, by the conclusion in  Step 3 and the first conclusion in \eqref{20241222-yb-DecomposedInformationOfSolutionAtInitialTime}, we obtain condition $(2)$ in $(ii)$ and
$\phi \in H^{r,s}_{loc,+}(-\tau,x_0)$.
Since $y|_{(-\tau,0) \times \Omega}   =  \phi$, the above, together with the first conclusion in \eqref{20241222-yb-DecomposedInformationOfSolutionAtInitialTime}, implies condition $(1)$ in $(ii)$. Finally, condition $(3)$
in $(ii)$ follows from the conclusion in Step 4 and
the second conclusion in \eqref{20241222-yb-DecomposedInformationOfSolutionAtInitialTime}. Therefore, we conclude that all the conditions $(1)$, $(2)$, and $(3)$ are true, which contradicts the statement $(ii)$.  So, the statement $(i)$ is true.

\vskip 5pt
Hence, we complete the proof of Theorem \ref{20240406-yubiao-Theorem-SingularitiesOfHistoricalValueAtEndpoints}.
\end{proof}

\section{Numerical simulations}\label{section-NumericalSimulations}

In this section, some numerical experiments are presented for
equation \eqref{20220510-yb-OriginalSystemWithTimeDelay} (with the space dimension $n=1$) to demonstrate the propagation of singularities in Theorem \ref{20230510-yb-theorem-PropagationOfSingularitiesOnlyWithInitialdata}.

Let $\Omega=(0,1)$, $\tau=1$  and $a=1$ in equation \eqref{20220510-yb-OriginalSystemWithTimeDelay}. We consider
\begin{align*}%\label{202320-yb-ModelWithConstantMemoryKernel}
    y'(t)  - \Delta_D y(t) - y(t-1) =0, ~t>0;
    ~~  y(0)= \delta_{0.3},
    ~ y|_{(-1, 0)} \equiv 0,
\end{align*}
where $\delta_{0.3}$ is the Dirac measure at $x=0.3$. Let
\begin{align*}
   j_t := [t],~ t\geq 0,
\end{align*}
(where $[t]$ denotes the largest number that is not larger than $t$).
By Theorem \ref{20230510-yb-proposition-ExpressionForEvolutionWithTimeDelay}, (v) of Remark \ref{20250118-yb-remark-ExpressionOfSolutions} and the spectral method, we obtain  that in the space $C([0,+\infty); \mathcal{H}^{-1})$,
\begin{align*}
    y(t,x; \delta_{0.3}, 0)
    = \sum_{k=1}^{+\infty}
    \sum_{j=0}^{j_t}
        \frac{ (t-j\tau)^j }{ j! }
          e^{-\lambda_k (t-j\tau)}
    e_k(0.3) e_k(x),
    ~t \geq 0, ~ x \in (0,1),
\end{align*}
with the convention that  $0^0:=1$. Here,  $\{\lambda_k\}_{k\geq 1}$ and $\{e_k\}_{j\geq 1}$ are the eigenvalues and the corresponding eigenfunctions (normalized in $L^2(0,1)$) of the operator $-\Delta_D$, respectively.

One can easily check that the right limit of $\partial_t^{j_t}  y(t,\cdot; \delta_{0.3}, 0)$ exists in $\mathcal{H}^{-1}$ at each time $t\geq0$. We denote this limit by $\partial_t^{j_t}  y(t+0,\cdot; \delta_{0.3}, 0)$.
Then we can directly verify that
\begin{align}\label{20250105-yb-ExpressionOfTimeDerivativesOfSolutions}
    \partial_t^{j_t}  y(t+0,\cdot; \delta_{0.3}, 0)
= \sum_{k=1}^{+\infty}
    \sum_{j=0}^{j_t}
        \sum_{l=0}^{ j }
        C_{j_t}^l
        \frac{ (t-j\tau)^{j-l} }{ (j-l)! }
        (-\lambda_k)^{j_t - l}
          e^{-\lambda_k (t-j\tau)}
    e_k(0.3) e_k(\cdot),
    ~ t \geq 0.
\end{align}

Next, we take the first $60$ frequency components in \eqref{20250105-yb-ExpressionOfTimeDerivativesOfSolutions}, discretize $\Omega=(0,1)$ with the mesh size $1/300$, and then draw the function $\partial_t^{j_t}  y(t+0,\cdot; \delta_{0.3}, 0)$ at different time instants in Figure  \ref{202501-yb-SolutionsWithWaves}.
In each subfigure, the blue curve represents the function $\partial_t^{j_t}  y(t+0,\cdot; \delta_{0.3}, 0)$ at the time instant specified there.
\begin{figure}[!htb]
\label{202501-yb-SolutionsWithWaves}
\centering
\includegraphics[height=10cm,width=15cm]{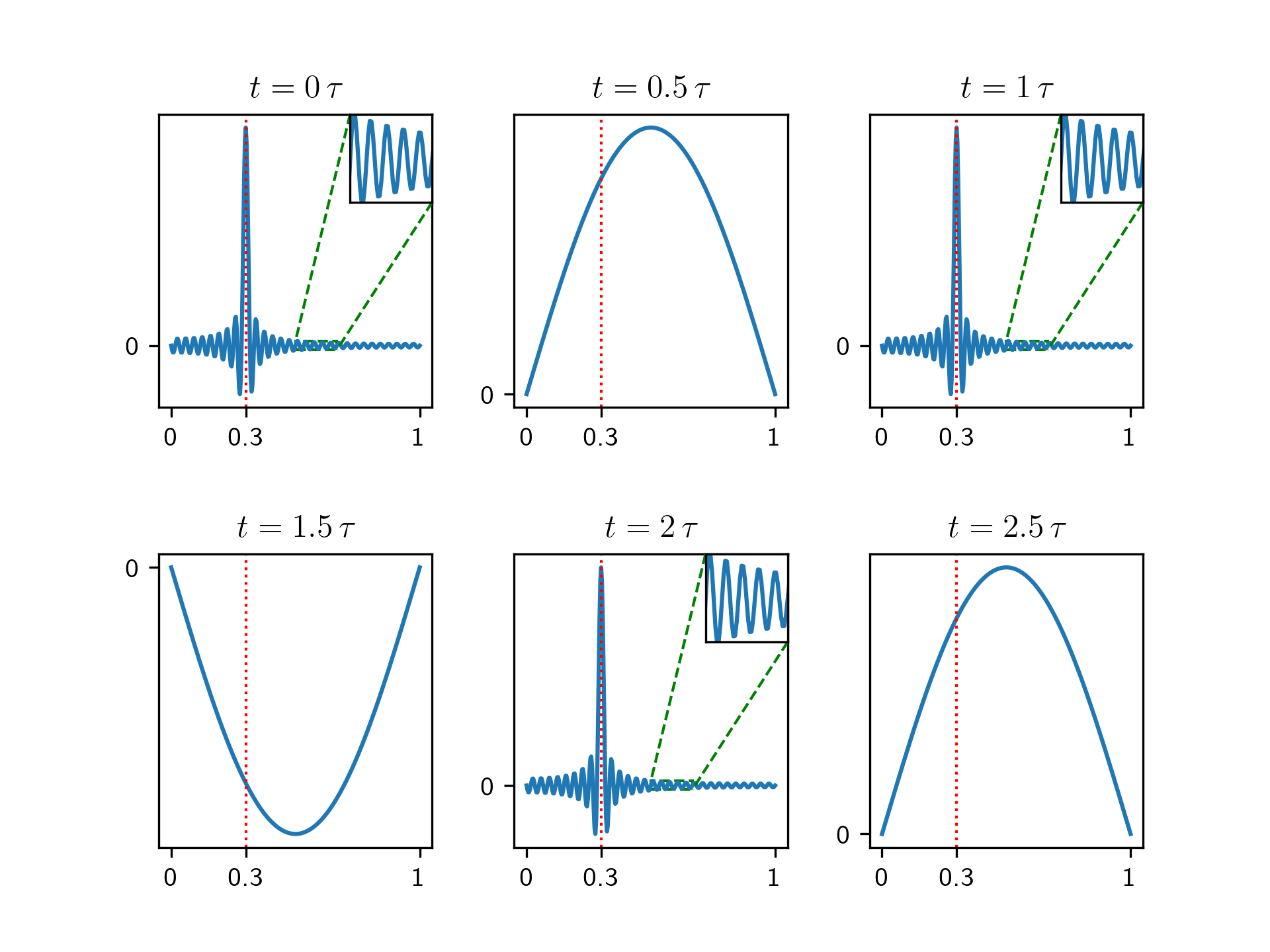}
\caption{Images of $\partial_t^{j_t}  y(t+0,\cdot; \delta_{0.3}, 0)$ (drawn in blue) at six time instants.}
\end{figure}

From these subfigures we see that: $(i)$ at the time instants $t=\tau, 2\tau$, the function $\partial_t^{j_t}  y(t+0,\cdot; \delta_{0.3}, 0)$ reproduces the singularity of the initial datum $\delta_{0.3}$ at the same position ($x=0.3$); $(ii)$ at the other three positive time instants, the function $\partial_t^{j_t}  y(t+0,\cdot; \delta_{0.3}, 0)$ has no singularities. These confirm the propagation of singularities along the time direction for equation \eqref{20220510-yb-OriginalSystemWithTimeDelay} (introduced in Theorem \ref{20230510-yb-theorem-PropagationOfSingularitiesOnlyWithInitialdata}).

\section{Appendices}\label{section-yu-section-2-22}

% \subsection{An interpolation property}

% \begin{lemma}\label{20241218-yb-lemma-InterpolationSpace}
%     Let $h \in \mathcal{D}'(\mathbb R^{n+1})$. Let $r,s,j\in \mathbb N$. Then, the following two statements are equivalent:
%     \begin{itemize}
%         \item[$(i)$] $h \in H^{r+j, s}$
%     \end{itemize}
% \end{lemma}

\subsection{Proof of Lemma \ref{20230524-yb-lemma-PropagationOfSmoothnessForHeatSolutions}}

\begin{proof}[Proof of Lemma \ref{20230524-yb-lemma-PropagationOfSmoothnessForHeatSolutions}]
We arbitrarily fix $x_0 \in U$ and then take an open ball $B_r(x_0)$, centered at $x_0$ and of radius $r>0$, such that $B_{2r}(x_0) \Subset U$.
We take a cut-off function $\rho \in C_0^{\infty}(B_{2r}(x_0))$ with $\rho \equiv 1$ over $B_{r}(x_0)$. Then it is clear that
\begin{align*}
    \rho y_0  \in C_0^{\infty}(B_{2r}(x_0))
    ~\text{and}~
    \rho f  \in C^{\infty}\big( [0,T); C^{\infty}_0(B_{2r}(x_0)) \big).
\end{align*}
By \cite[Theorem 5, Page 382]{Evans-2010-AMS}, there is a function $\hat y \in  C^\infty([0,T) \times B_{2r}(x_0))$
such that
\begin{align*}
\left\{
    \begin{array}{ll}
     \partial_t \hat y  -  \Delta  \hat  y  =  \rho f
     &~\text{in}~
     (0, T) \times B_{2r}(x_0),
     \\
     \hat y  = 0  &~\text{on}~
     (0, T) \times \partial B_{2r}(x_0),
     \\
     \hat y|_{t=0}  = \rho y_0  &~\text{in}~  B_{2r}(x_0).
    \end{array}
\right.
\end{align*}
Since $\rho \equiv 1$ over $B_{r}(x_0)$,  with \eqref{20241220-yb-InitialDataAndNonhomogeneousTerm}, the above equality implies
\begin{align*}
    \partial_t (y - \hat y)   - \Delta (y - \hat y) = 0
    ~\text{over}~
    (0, T) \times B_{r}(x_0);
    ~~
     (y - \hat y)|_{t=0}  = 0   ~\text{in}~  B_{r}(x_0).
\end{align*}
We write $\tilde{y}$ for the zero extension of $y-\hat y$ over $(-\infty, T) \times B_{r}(x_0)$. Direct  computations yield
\begin{align*}
    \partial_t \tilde{y}   - \Delta \tilde{y} = 0
    ~\text{over}~
    (-\infty, T) \times B_{r}(x_0).
\end{align*}
At the same time, since the heat operator $\partial_t  - \Delta $ is hypoelliptic (see \cite[Definition 11.1.2 and (iii) and (iv) of Theorem 11.1.1]{Hormander-2}),  we find from the last equality that
    $\tilde{y}   \in  C^{\infty} ( (-\infty, T) \times B_{r}(x_0))$.
In particular,
\begin{align*}
    (y - \hat y)|_{ [0,T) \times B_{r}(x_0) }
= \tilde{y}|_{ [0,T) \times B_{r}(x_0) }
\in  C^{\infty} ( [0,T) \times B_{r}(x_0)).
\end{align*}
Note that $\hat y \in C^\infty([0,T) \times B_{2r}(x_0))$. We conclude
\begin{align*}
    y = \hat y  +  (y - \hat y)     \in   C^{\infty} ( [0,T) \times B_{r}(x_0)).
\end{align*}
Because $x_0\in U$ was arbitrarily taken, the desired conclusion follows immediately. This completes the proof of Lemma \ref{20230524-yb-lemma-PropagationOfSmoothnessForHeatSolutions}.
\end{proof}

\subsection{Proof of statement \eqref{20250203-DistributionEmbeddedIntoSobolevSpaces}}
\label{20250203-DistributionsAndSobolevSpaces}

The statement \eqref{20250203-DistributionEmbeddedIntoSobolevSpaces} tells that a distribution is locally in a Sobolev space. Its proof is presented below for sake of  the  completeness of this paper.

\begin{proof}[Proof of \eqref{20250203-DistributionEmbeddedIntoSobolevSpaces}]
Arbitrarily fix a cut-off function $\rho \in C_0^{\infty}(U_1)$. Note that $U_1 \Subset U$. It is clear that  $\rho y_0 $ is compactly supported in $U_1$.
By the definition of distributions (see for instance \cite[Definition 2.1.1, page 33]{Hormander-1}), there is a constant $C>0$ and a $k\in \mathbb N^+$ such that for each $\varphi \in C_0^\infty(U)$,
\begin{align*}
   | \langle \rho y_0, \varphi \rangle_{\mathcal{D}'(U), C_0^\infty(U)} |
\leq   C \| \varphi \|_{C^k(U_1)}.
\end{align*}
Then, by the Sobolev embedding theorem (see for instance \cite[Theorem 6, Section 5.6, Page 270]{Evans-2010-AMS}), there is a $C_1>0$ such that for each $\varphi \in C_0^\infty(U)$,
\begin{align*}
   | \langle \rho y_0, \varphi \rangle_{\mathcal{D}'(U), C_0^\infty(U)} |
\leq   C_1 \| \varphi \|_{H^{k + n}(U)}.
\end{align*}
% Since $y_0\in \mathcal{D}'(U)$, the distribution $ \rho y_0  $ is compactly supported in $U$.
This implies that the zero extension of  $\rho y_0 $ over $\mathbb R^n$ is in the Sobolev space $H^{-(k+n)}(\mathbb R^n)$. By the arbitrariness of $\rho$ in $C_0^{\infty}(U_1)$, we conclude that $y_0 \in H^{-(k+n)}_{loc}(U_1)$.  This ends the proof.
\end{proof}

\subsection{Two lemmas on regularity}

\begin{lemma}\label{20241230-yb-lemma-InterEquivalenceBetweenTwoSobolevSpaces}
Let $\mathcal{E}'(\Omega)$ be the space of distributions compactly supported over $\Omega$. Then, for each $s\in \mathbb R$,
\begin{align}\label{20241229-yb-IdentificationOfTwoFunctionSpaces}
   H^s_{loc}(\Omega)  \cap \mathcal{E}'(\Omega)
   = \mathcal{H}^s \cap  \mathcal{E}'(\Omega).
\end{align}
\end{lemma}

\begin{proof}
We arbitrarily fix a compact non-empty subset $K\Subset \Omega$.
For each $s\in \mathbb R$, we define
\begin{align*}
    H^s_K(\Omega) :=    \big\{
        f\in H^s_{loc}(\Omega) ~:~  \text{supp}\, f \subset K
    \big\}
    ~\text{and}~
    \mathcal{H}^s_K := &  \big\{
        f\in \mathcal{H}^s ~:~  \text{supp}\, f \subset K
    \big\}.
\end{align*}
 To prove \eqref{20241229-yb-IdentificationOfTwoFunctionSpaces}, it suffices to show that
\begin{align}\label{20241230-yb-EquivalenceOfTwoFunctionSpaces}
    H^s_K(\Omega)  =  \mathcal{H}^s_K,
    ~\forall\,  s\in \mathbb R.
\end{align}
The rest of the proof is divided into the following two steps.

\vskip 5pt
\noindent\textit{
Step 1. We prove that for each $s \geq 0$, the identity operator is an isomorphism from $H^{s}_{K}(\Omega)$
to $\mathcal{H}^s_K$.
}

Note that when $s_2 \geq s_1 \geq 0$,
\begin{align*}
    H^{s_2}_{K}(\Omega) \subset H^{s_1}_{K}(\Omega)
    ~\text{and}~
    \mathcal{H}^{s_2}_K  \subset \mathcal{H}^{s_1}_K.
\end{align*}
By this and the interpolation theorem (see \cite[Theorem 5.1]{Lions-Magenes}), it suffices to prove that for each integer $m\in \mathbb N$, the identity operator  is an isomorphism from $H^{2m}_{K}(\Omega)$  to $\mathcal{H}^{2m}_K$.

Let $m \in \mathbb N$. We first show that
\begin{align}\label{20241230-yb-SubsetRelationBetweenTwoSobolevSpaces}
    H^{2m}_{K}(\Omega)  \subset \mathcal{H}^{2m}_K.
\end{align}
 To this end, we arbitrarily fix $f \in H^{2m}_{K}(\Omega)$. Then we have
\begin{align*}
  \text{supp}\, f \subset K
  ~\text{and}~
  \Delta_D^m f = \Delta^m f   \in  L^2(\Omega),
\end{align*}
where $\Delta$ is the Laplacian operator without any boundary condition. This implies $f\in \mathcal{H}^{2m}_K$, which leads to \eqref{20241230-yb-SubsetRelationBetweenTwoSobolevSpaces}.

Next,  we will prove
\begin{align}\label{20241230-yb-SupsetRelationBetweenTwoSobolevSpaces}
     \mathcal{H}^{2m}_K\subset H^{2m}_{K}(\Omega) .
\end{align}
Let $f\in \mathcal{H}^{2m}_K$. It is clear that
\begin{align*}
  \text{supp}\, f \subset K
  ~\text{and}~
  \Delta^m f  = \Delta_D^m f  \in  L^2(\Omega).
\end{align*}
This indicates that $f\in H^{2m}_K(\Omega)$, which yields \eqref{20241230-yb-SupsetRelationBetweenTwoSobolevSpaces}.

Finally, from \eqref{20241230-yb-SubsetRelationBetweenTwoSobolevSpaces} and \eqref{20241230-yb-SupsetRelationBetweenTwoSobolevSpaces}, we obtain
    $H^{2m}_{K}(\Omega)  = \mathcal{H}^{2m}_K$.
At the same time, the norms $\| \cdot \|_{H^{2m}(\Omega)}$ and $\|\cdot\|_{\mathcal{H}^{2m}}$ are complete in the spaces $H^{2m}_{K}(\Omega)$ and $\mathcal{H}^{2m}_K$, respectively. Hence, we conclude that these norms are equivalent and that the identity operator  is an isomorphism from $H^{2m}_{K}(\Omega)$  to $\mathcal{H}^{2m}_K$.

\vskip 5pt
\noindent\textit{
Step 2. We show \eqref{20241230-yb-EquivalenceOfTwoFunctionSpaces} for general $s\in \mathbb R$.
}

When $s \geq 0$, \eqref{20241230-yb-EquivalenceOfTwoFunctionSpaces} follows directly from Step 1.

When $s<0$, one can check that the spaces $H^{s}_{K}(\Omega)$ and $\mathcal{H}^{s}_K$ are dual spaces of the spaces $H^{-s}_{K}(\Omega)$ and $\mathcal{H}^{-s}_K$, respectively. At the same time, the latter two spaces are the same by Step 1 (with $s$ replaced by the positive number $-s$ here).
Then, the former two spaces $H^{s}_{K}(\Omega)$ and $\mathcal{H}^{s}_K$ coincide. Thus, \eqref{20241230-yb-EquivalenceOfTwoFunctionSpaces} is true for $s<0$. This ends the proof.
\end{proof}

The following result is classical, and the proof is provided here for the sake of completeness.

\begin{lemma}\label{20241223-yb-lemma-SolutionsWithNonhomogeneousTermInFunctionalSetting}
Let $T>0$ and $f\in L^2( (0,T); \mathcal H^s )$ with $s\in \mathbb R$. Then, the solution to the following equation
\begin{align*}%\label{20230519-yb-EquationWithZeroIntialDataInMainProof}
\left\{
    \begin{array}{ll}
        \partial_t y - \Delta y = f &\text{in}  ~ (0,T) \times \Omega,\\
        y = 0  &\text{on}  ~(0,T)  \times \partial\Omega,\\
        y|_{t=0}= 0  &\text{in}   ~\Omega
    \end{array}
\right.
\end{align*}
satisfies
$y \in H^1((0,T); \mathcal H^{s}) \cap L^2( (0,T); \mathcal H^{s+2})$. Furthermore, we have $ y \in C([0,T]; \mathcal{H}^{s+1})$.
\end{lemma}

% \begin{proof}
%     \eqref{def-space-with-boundary-condition}.
% \end{proof}

\begin{proof}
The proof is standard by the spectral method.
Recall that $\{\lambda_i\}_{i\geq 1}$ and $\{ e_i \}_{i\geq 1}$ are the eigenvalues and corresponding eigenfunctions of $-\Delta_D$, respectively. Let
\begin{align*}
    f(t) = \sum_{i=1}^{+\infty}f_i(t)e_i, ~t \in (0,T)
    ~\text{for some}~
    (\lambda_i^{s/2} f_i)_{i\geq 1} \in L^2((0,T); l^2).
\end{align*}
Then it follows that
\begin{equation}\label{yu-9-19-102}
    y(t,\cdot) = \sum_{i=1}^{+\infty} \left(
    \int_0^t  e^{-\lambda_i \sigma}  f_i (t-\sigma) d\sigma
    \right)  e_i(\cdot),   ~t\in[0,T].
\end{equation}
Let $g_1 * g_2 (t) := \int_0^t g_1(t)  g_2(t-s) ds$, $t\geq 0$. We obtain from \eqref{yu-9-19-102} that
\begin{align*}
    \| y \|_{ L^2( (0,T); \mathcal H^{s+2} ) }^2
=&  \sum_{i \geq 1}  \lambda_i^{s+2}  \int_0^T  \bigg(
        \int_0^t  e^{-\lambda_i \sigma}  f_i (t-\sigma) d\sigma
    \bigg)^2   dt
    \nonumber\\
=& \sum_{i \geq 1}    \big\|
       % \big[
            (\lambda_i e^{-\lambda_i \cdot})  *   (\lambda_i^{s/2}f_i)
        %\big] (t)
    \big\|_{L^2((0,T))} ^2
    \nonumber\\
\leq&  \sum_{i \geq 1}  \Big(
        \| \lambda_i e^{-\lambda_i \cdot} \|_{L^1((0,T))}   \cdot
        \| \lambda_i^{s/2}f_i \|_{L^2((0,T))}
    \Big)^2
    \nonumber\\
\leq&  \sum_{i \geq 1}     \int_0^T     \lambda_i^{s}       |f_i(t)|^2  dt
= \| f\|_{ L^2( (0,T); \mathcal H^{s} ) }^2
    < + \infty.
\end{align*}
Thus, $y \in L^2( (0,T); \mathcal H^{s+2})$. Since $\partial_t y = \Delta y + f$ and $f \in L^2( (0,T); \mathcal H^{s})$, we then get
$y \in H^1((0,T); \mathcal H^{s})$.
At last, because $y \in L^2( (0,T); \mathcal H^{s+2})$ and $\partial_t y \in L^2( (0,T); \mathcal H^{s})$, we determine
$y \in C([0,T]; \mathcal{H}^{s+1})$.
This finishes the proof.
\end{proof}

\section{Acknowledgments}

The authors gratefully thank Prof. Karl Kunisch for his invaluable discussions and comments on this paper.

The authors were partially supported by the National Natural Science Foundation of China under grants 12171359 and 12371450. The author Y. Zhang was also funded by the Humboldt Research Fellowship for Experienced Researchers program from the Alexander von Humboldt Foundation.

%\bigskip\bigskip
%\noindent\textbf{Data Availability Statement} Data sharing is not applicable to this paper because no
%    dataset was analysed or generated  during the study.
%
%\bigskip
%\noindent\textbf{Statements and Declarations}
%
%\noindent\textbf{Competing Interests} The authors declare that  they have no financial interests nor any conflict of interest.

\end{document}